\documentclass[11pt]{article}
\usepackage{geometry}
\usepackage{soul}
\usepackage{tikz,enumerate,enumitem, amsmath, mathtools, amsfonts, amssymb, amsthm, fontawesome, tikz-cd, mathdots, verbatim}
\usepackage[colorlinks=true,pdfpagemode=none,urlcolor=blue,linkcolor=blue,citecolor=blue]{hyperref}
\usepackage{xcolor}
\usetikzlibrary{calc} 
\mathtoolsset{showonlyrefs}

\newtheorem{thm}{Theorem}
\newtheorem{cor}{Corollary}[section]
\newtheorem{lem}{Lemma}[section]
\newtheorem{prop}{Proposition}[section]
\newtheorem{thmlit}{Theorem}[section]
\newtheorem*{prop*}{Proposition}

\theoremstyle{definition}

\theoremstyle{remark}
\newtheorem{rmk}{Remark}

\newcommand{\R}{\mathbb{R}}
\newcommand{\C}{\mathbb{C}}

\newcommand{\N}{\mathbb{N}}
\newcommand{\Z}{\mathbb{Z}}

\newcommand{\I}{\mathcal{I}}
\newcommand{\Sc}{\mathcal{S}}
\newcommand{\E}{\mathcal{E}}
\newcommand{\F}{\mathcal{F}}
\newcommand{\A}{\mathcal{A}}

\newcommand{\MO}{\mathcal{O}}

\newcommand{\vphi}{\varphi}

\DeclareMathOperator{\supp}{Supp}
\DeclareMathOperator{\Span}{Span}
\DeclareMathOperator{\Ker}{Ker}
\DeclareMathOperator{\Real}{Re}

\DeclareMathOperator{\coDim}{codim}
\DeclareMathOperator{\Res}{Res}
\DeclareMathOperator{\Fa}{Fa}

\newcommand{\el}{\ell}

\def\XXint#1#2#3{{\setbox0=\hbox{$#1{#2#3}{\int}$}
     \vcenter{\hbox{$#2#3$}}\kern-.5\wd0}}

\title{Polyhomogeneous mapping properties of the Radon transform and backprojection operator on the unit ball}
\author{Seiji Hansen\thanks{Department of Mathematics, University of California, Santa Cruz, 95064; email:sekhanse@ucsc.edu}}

\begin{document}

\maketitle

\begin{abstract}
This article covers polyhomogeneous mapping properties of the Radon transform $R$ of smooth functions on the open unit ball $\Omega\subset\R^n$ and the back-projection operator $R^*$ on $Z=(-1,1)\times S^{n-1}\subset\R\times S^{n-1}$. We construct a double $b$-fibration which desingularizes the point-hyperplane relation of $\overline{\Omega}$ as the total space of a fibration over $\overline{Z}$. We provide formulas for $R$ and $R^*$ in operations generated by the associated $b$-fibrations and sharper estimates on the polyhomogeneous mapping properties of $R$ and $R^*$ compared to classic estimates using classic Mellin functional techniques. We include a discussion of a one (complex) parameter family of normal operators associated to $R$ mapping $C^{\infty}(\overline{\Omega})$ to itself.
\end{abstract}

\section{Introduction}\label{section:intro}

\hspace{20pt} This article studies mapping properties for the Radon transform and its adjoint (the backprojection operator) in classes of conormal distributions on the unit Euclidean ball of $\R^n$, specifically polyhomogeneous, conormal functions.

\hspace{20pt} In our setting, fix $n\geq 2$, and parameterize the affine hyperplanes of $\R^n$ as the set $\Pi(s,\theta) = s\theta + \theta^{\perp}$ for $(s,\theta)\in\R\times S^{n-1}$, by a normal vector $\theta$ and a signed distance $s$. This is a $2$-to-$1$ covering where $\Pi(s,\theta) = \Pi(-s,-\theta)$. The \textit{Radon transform} sends a Schwartz class function $u\in\Sc(\R^n)$ to the function defined by
	\begin{align}\label{eq:defn-Radon}
		Ru(s,\theta) = \int_{\Pi(s,\theta)}u(x)\ d\mu_{(s,\theta)}, \qquad (s,\theta)\in\R\times S^{n-1},
	\end{align}
where $\mu_{(s,\theta)}$ is the usual $(n-1)$-dimensional Lebesgue measure on $\Pi(s,\theta)$. A hyperplane is assigned to the integral over that hyperplane. Each $Ru$ is even with respect to the involution $(s,\theta)\mapsto(-s,-\theta)$. The $L^2$-adjoint of \eqref{eq:defn-Radon} under the usual $L^2$ spaces on $\R^n$ and $\R\times S^{n-1}$ is the \textit{backprojection operator} $R^*$ which for $v\in C^{\infty}(\R\times S^{n-1})$ has expression 
	\begin{align}\label{eq:defn-backproj}
		R^*v(x) = \int_{S^{n-1}}v(x\cdot\theta,\theta)\ d\theta, \qquad x\in\R^n,
	\end{align}
where $\cdot$ is the standard inner product on $\R^n$.

\hspace{20pt} Transforms \eqref{eq:defn-Radon} and \eqref{eq:defn-backproj} are classical objects of study in integral geometry (see, e.g., \cite{radon20051,Helgason1980,oberrlin1982mapping,ehrenpreis2003}) with broad applications that include X-ray computerized tomography ($n=2$) and magnetic resonance imaging ($n=3$); see, e.g. \cite{natterer2001,Louis1984,Maass1987,Maass1991,Maass1992}. Inversion formulas for $R$ involving $R^*$ as well as stability estimates on scales of Hilbert space over $\R^n$, motivated by the Fourier slice theorem, are well-known \cite{natterer2001,Sharafutdinov2017}. A study of mapping properties of $R$ and $R^*$ on integrands modeled by asymptotic expansions in $1/|x|$ was completed by A. Katsevich in \cite{Katsevich2001}.

\hspace{20pt} In practice, a body under examination or X-ray data $u$ is compactly supported and has a specific geometry that influences the analysis of $Ru$ and $R^*u$. In this article we make the assumption that $u$ has support contained in the manifold with boundary 
	\begin{align}\label{eq:OmegaBar}
    		\overline{\Omega} := \{x\in\R^n : |x|\leq 1\},\qquad \partial\Omega = \{x\in\R^n : |x| = 1\},
	\end{align}
in which case $Ru$ has support in the set of affine hyperplanes intersecting $\overline{\Omega}$, given by the manifold with boundary 
	\begin{align}\label{eq:Zbar}
		\overline{Z} := [-1,1]\times S^{n-1}, \qquad \partial Z = \{-1\}\times S^{n-1}\cup\{1\}\times S^{n-1}.
	\end{align} 
Singular Value Decompositions for this restriction have been established in \cite{Louis1984} for various weighted $L^2$ topologies, suggesting other spaces which may capture sharp mapping properties. To push the analysis further, a number of questions remain open, including: 
\begin{enumerate}
    \item[Q1:] Given subspaces $E\subset C^\infty(\Omega)$ and $F\subset C^\infty(Z)$ of polyhomogenous conormal type near the boundary, how can we describe $R(E)$, or $R^* (F)$ ?
    \item[Q2:] Can one find weight functions $w_1,w_2$ (possibly singular at $\partial\Omega$ and $\partial Z$) such that the operator $R^* w_2 R w_1$ maps $C^\infty(\overline{\Omega})$ into itself? Can such an operator be an isomorphism? If yes, can this isomorphism be made {\em tame} relative to some scale of Hilbert spaces on $\Omega$?
\end{enumerate}
The framing of questions Q1 and Q2 can be found in the recent topical review \cite{Monard2024} concerned with boundary issues and sharp mapping properties for the geodesic X-ray transform on convex, non-trapping manifolds, where partial answers and open questions are provided on this, along with further motivations for these questions (e.g., applications to Bayesian inversions in infinite-dimensional, non-parametric settings). 

\hspace{20pt} In this article, we answer question Q1, namely we discuss the forward mapping properties of $R$ and $R^*$ on polyhomogeneous conormal spaces on $\overline{\Omega}$ and $\overline{Z}$, see Theorems \ref{RadonExpansion} and \ref{MainResult}, respectively. We also discuss some consequences such as Corollary \ref{cor:AdjustedNormal}, which also partially address Q2. 

\hspace{20pt} Our methodology follows that of Mazzeo and Monard in \cite{Mazzeo2021}, where analogous results are derived for the geodesic X-ray transform. We start by examining the well-known double-fibration of the point-hyperplane relation $E$ of $\R^n$
\begin{align}\label{eq:doublefib}
    \R\times S^{n-1} \longleftarrow E \longrightarrow \R^n,      
\end{align}
and analyze the restriction of $E$ to $\overline{\Omega}\subset\R^n$ in order to realize $R$ and $R^*$ as pushforward-pullback compositions, see e.g. \cite{Helgason1980}. The restriction of \eqref{eq:doublefib} to $Z$ and $\Omega$ remains a double-fibration but the fiber degenerates near the boundaries. We first show that this boundary degeneracy can be desingularized into a {\em double b-fibration} in the sense of Melrose \cite{Melrose1992,GrieserGruber2001}. This is the appropriate setting to express the restricted operators $R, R^*$ as pushforward-pullback compositions by b-fibrations and conormal multiplication operators, and to deduce a first estimate on the forward mapping properties for $R, R^*$ in polyhomogeneous spaces, using the classical Pushforward and Pullback Theorems of Melrose (see Appendix \ref{PullbackPushforwardTheorems}). Such a first estimate is sharp for $R$ but not for $R^*$, as it is seen to overestimate index sets. For example, when $n$ is odd, the initial estimate predicts that $R^*(C^{\infty}(\overline{Z}))\subset C^{\infty}(\overline{\Omega}) + \rho^{n-1}\log(\rho)C^{\infty}(\overline{\Omega})$ (where $\rho:= 1-|x|^2$, see Proposition \ref{Prop:Back-ProjectionEstimate} and \ref{Prop:ForwardEstimate}) while one can observe from \eqref{eq:defn-backproj} that $R^* (C^{\infty}(\overline{Z}))\subset C^{\infty}(\overline{\Omega})$ for any $n$. We resolve this discrepancy by finding a proof of the latter estimate using Mellin transform methods following \cite{Mazzeo2021}. This method opens the machinery of the original Mellin transform-based proof of the Pushforward Theorem to find smaller index set estimates from pole-cancellation mechanisms in the Mellin-variable side. In the particular case of $R^*$, this mechanism removes some of the originally predicted index sets. 

\hspace{20pt} As a benchmark, since our case $n=2$ matches the case of the X-ray transform in the Euclidean disk, our main theorems recover the results of \cite{Mazzeo2021} in that case. On the other hand, our proofs also exhibit new phenomena which depend on the parity of the dimension $n$.

\hspace{20pt} \textbf{Outline.} We first state the main results in Section \ref{SectionMainResults} and corollaries that directly follow in Subsection \ref{SubsectionCorollaries}. In the proceeding Section \ref{SectionDevelopment}, we construct the double $b$-fibration $(G,\hat\pi,\hat P)$ and the pullback-pushforward formulae for $R$ and $R^*$ and conclude with proof of the main theorems in Subsections \ref{SubsectionProofRadon} and \ref{SubsectionProofBackproj}. Background information about manifolds with corners is available in Appendix \ref{ManifoldWithCorners} and \ref{bMaps}, conormal and polyhomogeneous functions spaces and polyhomogeneous index sets in Appendix \ref{AlgebraOnManifold}, and finally in Appendix \ref{PullbackPushforwardTheorems} one can find statements of Melrose's Pullback and Pushforward theorem and in Appendix \ref{MellinBackground} a reference to how Mellin Functionals in a boundary-defining coordinate can aide our study of index sets.

\section{Statement of main results and corollaries}\label{SectionMainResults}

\hspace{20pt} The manifolds with boundary (mwb) $\overline{\Omega}$ and $\overline{Z}$, defined by \eqref{eq:OmegaBar} and \eqref{eq:Zbar}, consist of one and two boundary hypersurfaces (bhs) respectively. The bhs of $\overline{\Omega}$ has the associated boundary-defining function (bdf) $\rho(x) = 1-|x|^2$. The two bhs of $\overline{Z}$ have associated bdf $(s,\theta)\mapsto 1\mp s$. The product $\sigma(s,\theta) = 1-s^2$ is a local bdf for $\partial Z$. The Beta function is defined by the integral formula
	\begin{align}
		B(\alpha,\beta) = \int_0^1x^{\alpha-1}(1-x)^{\beta-1}dx, \qquad \alpha,\beta\in\{z\in\C:\Real(z) > 0\}.
        \label{eq:beta}
	\end{align}
The Beta function extends to a meromorphic function on $\C^2$ by the identity $B(\alpha,\beta) = \frac{\Gamma(\alpha)\Gamma(\beta)}{\Gamma(\alpha+\beta)}$, where $\Gamma$ is the Gamma function. A function is of class $C^{\infty}_{even}(S^{n-1})$ if it is $C^{\infty}(S^{n-1})$ and fixed by the antipodal map $\theta\mapsto-\theta$. We follow the convention that the binomial coefficients ${m\choose -1}$ vanish if $m\neq-1$ and evaluate to $1$ when $m=-1$. The definition of polyhomogeneous (phg) component as well as for mwb, bhs, and (local) bdf can be found in Appendices \ref{AlgebraOnManifold} and \ref{ManifoldWithCorners}. 

\begin{thm}[Phg mapping properties of $R$]\label{RadonExpansion}
Let nonzero $a\in C^{\infty}(\partial\Omega)$ and $\tilde a(x) = a(\frac{x}{|x|})$ be its radial extension, $(\gamma,\el)\in\C\times\N_0$ with $\Real(\gamma) > -1$ and $\chi\in C^{\infty}_c[0,1)$ a cutoff $\chi\equiv1$ near $0$. For every $p\in\N_0$ there exists $A_p\in C^{\infty}(S^{n-1})$ such that the Radon transform of the phg component $u(x) = \tilde a(x)\ \rho^{\gamma}\log(\rho)^{\el}\ \chi(\rho)$ has expansion near $\sigma = 0$  
	\begin{align}\label{eq:RadonExpansion}
		Ru(s,\theta)\sim \sum_{m = 0}^{\infty}\sum_{k = 0}^{\el}\sigma^{\frac{n-1}{2} + \gamma + m}\log(\sigma)^k\ A_{m,l}(\theta;\gamma,k),
	\end{align}
where $\sim$ is asymptotic equivalence in the sense of \eqref{asymptoticequivalence} and for all $(m,k)\in\N_0\times\{0,1,...,\el\}$
	\begin{align}\label{eq:Radon-coefficients}
		A_{m,k}(\theta;\gamma,\el) = {\el\choose k}\sum_{p=0}^m\textstyle A_p(\frac{s}{|s|}\theta)\displaystyle{m-1\choose p-1}\frac{\partial^{\el-k}}{\partial\eta^{\el-k}}\Big|_{\eta = \gamma}B\Big(\eta + 1,\frac{n-1}{2} + p\Big).
	\end{align}
\end{thm}

\hspace{20pt} The phg mapping properties of $R^*$ are more complicated to state due to the interaction of the interior and boundary values of $u$ on $R^*u$. The function $R^*u|_{\partial\Omega}$ can be non-vanishing even while $u|_{\partial Z} = 0$. This property can be understood at the level of the incidence relation, as the restriction of the $b$-fibration $\hat P:G\rightarrow\overline{Z}$ to the corners $S\overline{Z}\cap G$ as being a fibration over $\overline{Z}$ instead of only fibering the boundary $\partial Z$ ($G$ is defined in Section \ref{SectionDevelopment} and $b$-fibrations in Appendix \ref{bMaps}). This allows interior values of $u\in C^{\infty}(Z)$ to influence boundary values of $\hat P^*u$ on the boundary faces $S\overline{Z}\cap G$. 

\hspace{20pt} For a discrete subset of indices $I\subseteq\C\times\N_0$, we denote by $\overline{I}$ the smallest smooth index set containing every $(\gamma,k)\in I$. We say $I$ generates the index set $\overline{I}$. For $I$ made of a single index $(\gamma,k)$, the index set $\overline{\big\{(\gamma,k)\big\}}$ is
    \begin{align}
        \overline{\big\{(\gamma,k)\big\}} = \big(\gamma + \N_0\big)\times\{0,1,...,k\}.
    \end{align}

\begin{thm}[Phg mapping properties of $R^*$]\label{MainResult}
Let nonzero $a\in C^{\infty}(S^{n-1})$, $(\gamma,\el)\in\C\times\N_0$ and $\chi\in C^{\infty}_c[0,1)$ a cutoff $\chi\equiv1$ near $0$ and set $u(s,\theta) = a(\theta)\ \sigma^{\gamma}\log(\sigma)^{\el}\ \chi(1\pm s)$ to be the associated phg component. Define $F_{(\gamma,\el)}$ in \textbf{case a} (pole cancellation) when $n\geq2$ is even and $\gamma = m\in\N_0$ by 
	\begin{align}\label{eq:min-index-1}
		F_{(\gamma,l)} = \overline{\bigg\{\bigg(\frac{n-1}{2} + m,\el-1\bigg), (0,0)\bigg\}},
	\end{align}
where $F_{(\gamma,\el)} = \overline{\big\{(0,0)\big\}}$ whenever $\el=0$, in \textbf{case b} when $n\geq 2$ is even and $\gamma\in\frac{1}{2} + \Z$ or \textbf{case c} (pole creation) when $n\geq2$ is odd and $\gamma\in-\N$ by 
    \begin{align}\label{eq:min-index-2}
		F_{(\gamma,\el)} = \overline{\bigg\{\bigg(\frac{n-1}{2} + \gamma,\el\bigg), (0,0), (\eta,\ell+1)\bigg\}}, \quad \text{where} \quad \eta = \max\left(0, \frac{n-1}{2} + \gamma\right),
	\end{align}
and in \textbf{case d} (generically) for all other $n\geq2$ and $(\gamma,\el)$ by
    \begin{align}\label{eq:min-index-3}
        F_{(\gamma,\el)} = \overline{\bigg\{\bigg(\frac{n-1}{2} + \gamma,\el\bigg), (0,0)\bigg\}}.
    \end{align}
Then $F_{(\gamma,\el)}$ is an index set for $R^*u$. 
\end{thm}

\subsection{Discussion of Theorem \ref{MainResult}}

\hspace{20pt} The index sets in Theorem \ref{MainResult} improve upon the classical index set upper estimates. When $n\geq2$ is even and $\gamma\in\N_0$, Proposition \ref{Prop:Back-ProjectionEstimate} estimates $F' = \overline{\big\{(\frac{n-1}{2} + \gamma,\el),(0,0)\big\}}$ which compared to \eqref{eq:min-index-1}, does not detect the reduction in the second component of the generating index to $(\frac{n-1}{2} + \gamma,\el-1)$.

\hspace{20pt} When $n\geq2$ is odd and $\gamma\in\Z$, the classical estimate listed in Proposition \ref{Prop:Back-ProjectionEstimate} expects $F'$ to follow \ref{eq:min-index-2} and estimates the index set $F' = \overline{\{(\frac{n-1}{2} + \gamma,l),(0,0),(\eta,l+1)\}}$ where $\eta = \max\{0,\frac{n-1}{2}+\gamma\}$. Instead, the $\gamma\in\N_0$ case (corresponding to phg components of functions of class $C^{\infty}(\overline{Z})$) is covered by the \eqref{eq:min-index-3} case which is without the third generator. This discrepancy is explained in Subsection \ref{SubsectionProofBackproj} as a consequence of pole structure of the kernel of an operator defined on $R^*u$ for $u$ a phg component of $\overline{Z}$ (compare equations \eqref{eq:pairing-prefinal}, \eqref{B-integral-formula}). Corollaries are discussed in the following Subsection \ref{SubsectionCorollaries}. In all other settings for $n$ and $\gamma$, Theorem \ref{MainResult} agrees with the classical estimates in Proposition \ref{Prop:Back-ProjectionEstimate}. 

\begin{rmk}
In each case of Theorem \ref{MainResult}, $\overline{\{(0,0)\}}$ appears in the index set. This is due to the nonlocality of $R^*$, for $\vphi\in C^{\infty}_c(Z)$ and phg $u\in C^{\infty}(Z)$ the decomposition $u =(1-\vphi)\ u +  \vphi\ u$, has second component $\vphi u\in C^{\infty}_c(Z)$ and $R^*(\vphi u)\in C^{\infty}(\overline{\Omega})$ which are the phg functions with index set $\overline{\{(0,0)\}}$. The remaining component is $(1-\vphi)u \equiv u$ in a neighborhood of $\partial Z$ and accounts for the indices arising from the boundary data of $u$. 
\end{rmk}

On the other hand, we can see immediately from Theorem \ref{MainResult}, the following lower estimate on the index set of $R^*u$ for $u$ a phg component.

\begin{cor}\label{cor:lowerbound}
Let $u$ be as in Theorem \ref{MainResult} for some $(\gamma,\el)\in\C\times\N_0$ and $E$ an index set of $R^*u$. Then $\overline{\{(\frac{n-1}{2} + \gamma,p)\}}\subseteq E$ where
    \begin{align}
        p = \begin{cases}
            \el-1, & (\gamma,\el)\text{ is case }a,\\
            \el+1 & (\gamma,\el)\text{ is case }b\text{ or }c\text{ and }\frac{n-1}{2}+\gamma\geq0\\
            \el, & \text{else},
        \end{cases}
    \end{align}
and coefficient $b = b_{(\frac{n-1}{2}+\gamma,p)}(\theta)$ of the phg component of $R^*u$ associated to $(\frac{n-1}{2}+\gamma,p)$ when $n$ is even is
    \begin{align}
        b(\theta) = S_{\frac{s}{|s|}}a(\theta)\ \displaystyle \omega_{n-2}\ \Gamma\Big(\frac{n-1}{2}\Big)\begin{cases}
            \el\ \Gamma(-\frac{n-1}{2}-\gamma)\ (-1)^{\gamma+1}\ \gamma! & \gamma\in\N_0,\\
            \frac{\Gamma(-\frac{n-1}{2}-\gamma)}{\Gamma(-\gamma)} & \gamma\in\frac{1}{2}+\Z,\ \frac{n-1}{2}+\gamma < 0\\
            \frac{(-1)^{\frac{n-1}{2}+\gamma + 1}}{(\el + 1)\ (\frac{n-1}{2}+\gamma)!\ \Gamma(-\gamma)} & \gamma\in\frac{1}{2}+\Z,\ \frac{n-1}{2}+\gamma \geq 0,\\
            \frac{\Gamma(-\frac{n-1}{2}-\gamma)}{\Gamma(-\gamma)}, & \text{else},
        \end{cases}
    \end{align}
and when $n$ is odd is 
    \begin{align}
        b(\theta) = S_{\frac{s}{|s|}}a(\theta)\ \displaystyle \omega_{n-2}\ \Gamma\Big(\frac{n-1}{2}\Big)\begin{cases}
            \frac{-1}{(\el+1)}{\prod'}_{k=0}^{\frac{n-3}{2}}\frac{-1}{\frac{n-1}{2} + \gamma-k} & \gamma\in-\N,\ \frac{n-1}{2}+\gamma \geq 0,\\
            \prod_{k=0}^{\frac{n-3}{2}}\frac{-1}{\frac{n-1}{2} + \gamma-k}, & \text{else},
        \end{cases}
    \end{align}
where $S_{\frac{s}{|s|}}$ is the operator on $C^{\infty}(S^{n-1})$ defined by $S_{\frac{s}{|s|}}a(\theta) = a(\frac{s}{|s|}\theta)$, $\omega_{n-2}$ is the volume of the standard round, unit $S^{n-2}$, and $\prod'$ denotes the product of terms excluding the factor corresponding to $k=-(\frac{n-1}{2}+\gamma)$.
\end{cor}
\begin{proof}
The proof of the coefficient formulas are found in Subsection \ref{SubsectionProofBackproj}. Each coefficient is non-zero, with the exception of the $n$ even, $\gamma\in\N_0$, and $\el=0$ case, such that the index $(\frac{n-1}{2}+\gamma,p)$ is supported and so must the index set it generates. In the exception case, the index set is $\overline{\{(\frac{n-1}{2}+\gamma,-1)\}} = \emptyset$. 
\end{proof}

\subsection{Corollaries for normal operators constructed from $R$ and $R^*$}\label{SubsectionCorollaries}

From Theorem \ref{RadonExpansion} or Proposition \ref{Prop:ForwardEstimate}, when $n$ is even, $R\ C^{\infty}(\overline{\Omega})\subset \sigma^{1/2}\ C^{\infty}_{\text{even}}(\overline{Z})$. The appearance of a non-smooth $\sigma^{1/2}$ factor will require that under $R^*$, a function in $R\ C^{\infty}(\overline{\Omega})$ has index set indicated by case c (pole creation) of Theorem \ref{MainResult}. The $\sigma^{1/2}$-singularity will propagate as a $\log(\rho)$-singularity in $R^*R\ C^{\infty}(\overline{\Omega})$. 

\begin{cor}\label{cor:Normal}
Let $n$ be even. For each $\el\in\N$, 
	\begin{align}\label{eq:normalmap}
		(R^*R)^{\el}:C^{\infty}(\overline{\Omega})\rightarrow\sum_{0\leq j\leq k\leq\el}\rho^{(n-1)k}\log(\rho)^j\ C^{\infty}(\overline{\Omega})
	\end{align}
and there exists functions in the range of $(R^*R)^{\el}$ having non-vanishing components with $\log(\rho)^{\el}$ in their phg expansion. 
\end{cor}

\begin{rmk}
This generalizes \cite[Corollary 2.6]{Mazzeo2021} to the Radon transform setting.
\end{rmk}

\begin{proof}
By linearity of $R^*R$, induction, and the identity
    \begin{align}\label{eq:phg-identity}
        \A^{\overline{\{(\gamma,\el)\}}}_{\text{phg}}(\overline{\Omega}) = \sum_{j = 0}^{\el}\log(\rho)^j\rho^{\gamma}\ C^{\infty}(\overline{\Omega}),
    \end{align}
we only need to check that components with $(\gamma,j)\in(n-1)\N_0\times\N_0$ map to a component with $(\gamma + n-1,j+1)$. Let $j, k\in\N_0$ and $j\leq k$ and let $u\in\A^{\overline{\{((n-1)k,j)\}}}_{\text{phg}}(\overline{\Omega})$. By Theorem \ref{RadonExpansion} or Proposition \ref{Prop:ForwardEstimate}, $Ru$ will have index set $(\frac{n-1}{2},0) + \overline{\{((n-1)k,j)\}} = \overline{\{(\frac{2k+1}{2}(n-1),0)\}}$. Since $n$ is even, $Ru\in\sigma^{1/2}\ C^{\infty}_{\text{even}}(\overline{Z})$ and each index of $Ru$ falls under case b of Theorem \ref{MainResult}. It follows that $R^*Ru$ has index set $\overline{\{((k+1)(n-1),j + 1),(0,0)\}}$ corresponding to 
    \begin{align}
        R^*Ru\in C^{\infty}(\overline{\Omega}) + \A_{\text{phg}}^{\overline{\{((k+1)(n-1),j + 1)\}}}(\overline{\Omega}),
    \end{align}
from which \eqref{eq:normalmap} follows. 

\hspace{20pt} The range of $(R^*R)^{\el}$ will contain functions with non-vanishing component at index $((n-1)k,j)$ for all $0\leq j\leq k\leq\el$ due to the determination of the non-vanishing coefficient of such components in Theorem \ref{MainResult}.
\end{proof}

\begin{rmk}
When $n\geq2$ is odd, $R^*R$ has the mapping property $R^*R:C^{\infty}(\overline{\Omega})\rightarrow C^{\infty}(\overline{\Omega})$.
\end{rmk}

We can introduce a singular weight between $R$ and $R^*$ to correct for the non-smoothness of $Ru$ in even dimension $n$. The Singular Value Decomposition of these operators were previously studied in \cite{Louis1984,Mishra2023}.

\begin{cor}\label{cor:AdjustedNormal}
For any $n\geq2$ and $\gamma\in\C$ with $\Real(\gamma)\geq0$, $R^*\sigma^{-\frac{n-1}{2} + \gamma}R\rho^{\gamma}$ has the mapping property 
    \begin{align}\label{eq:adjusted-normal}
        R^*\sigma^{-\frac{n-1}{2} - \gamma}R\rho^{\gamma}: C^{\infty}(\overline{\Omega})\rightarrow C^{\infty}(\overline{\Omega})
    \end{align}
where $\sigma^{-\frac{n-1}{2} -\gamma}$ and $\rho^{\gamma}$ are multiplication operators with symbol given by their notation.  
\end{cor}

\begin{rmk}
This generalizes \cite[Lemma 2.9]{Mazzeo2021} to the Radon transform setting and Corollary \ref{cor:AdjustedNormal} also holds when $n$ is even and $\sigma^{-\frac{n-1}{2}}$ is replaced by $\sigma^{-\frac{1}{2}}$ like in that Lemma.
\end{rmk}

\begin{proof}
The functions $\rho^{\gamma}\in C^{\infty}(\Omega)$ and $\sigma^{-\frac{n-1}{2} - \gamma}\in C^{\infty}(Z)$ are non-vanishing and conormal and define shifts on each conormal space of $\Omega$ and $Z$ respectively. From either Theorem \ref{RadonExpansion} or Proposition \ref{Prop:ForwardEstimate}, $R\ \rho^{\gamma} C^{\infty}(\overline{\Omega})\subset\sigma^{\frac{n-1}{2} + \gamma}\ C^{\infty}_{\text{even}}(\overline{Z})$. It follows that $\sigma^{-\frac{n-1}{2}-\gamma}R\ C^{\infty}(\overline{\Omega})\subset C^{\infty}(\overline{Z})$. This containment will mean the non-pole creation case of Theorem \ref{MainResult} (either case a or case d depending on the parity of $n$) will apply to the mapping of a function of $\sigma^{-\frac{n-1}{2}-\gamma}R\ C^{\infty}(\overline{\Omega})$ under $R^*$. In either case, \eqref{eq:adjusted-normal} holds. 
\end{proof}

\section{Development of the main results}\label{SectionDevelopment}

\subsection{Geometry of the desingularization $G$}
The point-hyperplane relation associated to $R$,
    \begin{align}
        \I = \big\{(s,\theta,x)\in\R\times S^{n-1}\times\R^n : x\in\Pi(s,\theta)\big\},
    \end{align}
is a double fibration of $\I$ with fibrations being the canonical coordinate projection maps (see \cite{Gelfand1969,Helgason1980}). The fibration $P:\I\rightarrow\R\times S^{n-1}$ is a rank $(n-1)$ vector bundle\footnote{This could be viewed as the tautological bundle on the space of affine hyperplanes in $\R^n$.} and $\pi:\I\rightarrow\R^n$ is an $(n-1)$-sphere bundle. Consider the restriction of the coordinate projections to 
    \begin{align}\label{eq:incidence-rel}
        E = (Z\times\Omega)\cap\I, && \overline{E} = (\overline{Z}\times \overline{\Omega})\cap\I.
    \end{align}
We denote restrictions of each projection with the same symbol. The restriction to $E$ also forms a double fibration, except $P:E\rightarrow Z$ is an open $n$-ball bundle.\footnote{This is also sometimes called an $n$-disc bundle.} The restriction of $P$ to $\overline{E}$ fails to be a fibration. The fiber of $P$ over $(s,\theta)\in \overline{Z}$ is a radius $\sqrt{1-s^2}$ closed $n$-ball that is singular at $s^2 = 1$. These `small fibers' above the boundary $\partial Z$ are the analogue of `short geodesics' of \cite{Mazzeo2021} on manifolds with strictly-convex boundary.
    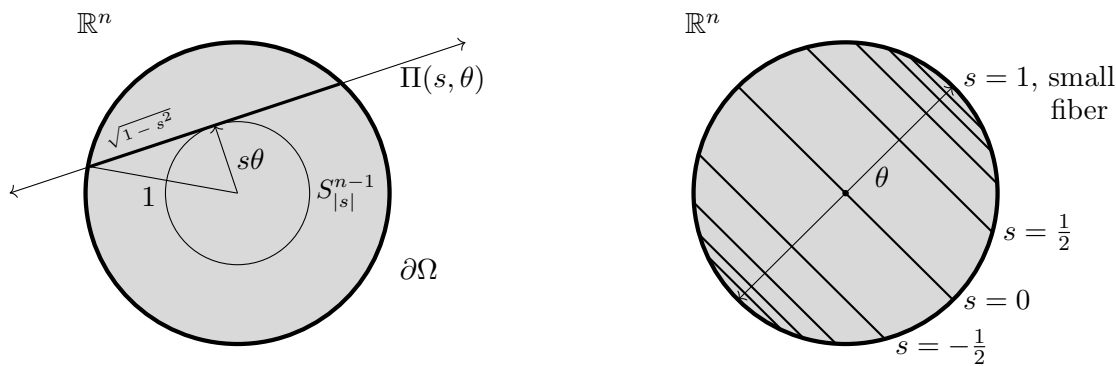
\begin{figure}[h]
        \centering
        \begin{tikzpicture}
            \draw[ultra thick,fill=gray,fill opacity=0.3] (-8,0) circle (2);
		      \draw [<->] (-11,0) -- (-5,2) node[pos=.3,above,rotate=20] {\tiny$\sqrt{1-s^2}$};
            \draw node at (-5.3,1.5) {$\Pi(s,\theta)$};
            \draw[-] (-9.94317,0.352277) -- (-8,0) node[pos=0.4,below] {$1$};
            \draw [-,very thick] (-9.94317,0.352277) -- (-6.65683,1.44772);
            \draw[->] (-8,0) --  (-8.3,0.9) node[pos=0.5,right] {$s\theta$};
            \draw (-8,0) circle (0.9486);
            \draw node[right] at (-7.1,0) {\small$S^{n-1}_{|s|}$};
            \draw node[right] at (-6,-1) {$\partial\Omega$};
            \draw node[left] at (-9.5,2.25) {$\R^n$};
            \draw[ultra thick,fill=gray,fill opacity=0.3] (0,0) circle (2);
            \draw node[left] at (-1.5,2.25) {$\R^n$};
            \draw[fill=black] (0,0) circle (1pt);
            \draw[<->] (-1.414,-1.414) -- (1.414,1.414) node[pos=0.67,below] {$\theta$};
            \draw[-,thick] (-1.414,1.414) -- (1.414,-1.414);
            \foreach \y in {.707,-0.707,1.06066017178,-1.06066017178,1.23743686708,-1.23743686708,1.32582521472,-1.32582521472}
                \draw[thick] ($(\y,\y)+sqrt(2-\y*\y)*(1,-1)$) -- ($(\y,\y)+sqrt(2-\y*\y)*(-1,1)$);
            \draw node[right] at (1.414,-1.414) {$s=0$};
            \draw node[right] at (1.93185165258,-0.5176380902) {$s=\frac{1}{2}$};
            \draw node[right] at (0.5176380902,-2.03185165258) {$s=-\frac{1}{2}$};
            \draw node[right] at (1.414,1.514) {$s=1$, small};
            \draw node[right] at (2.564,1.114) {fiber};
            \end{tikzpicture}
        \caption{Schematic of $\overline{\Omega}$. Decomposition of $\overline{\Omega}$ along $\theta$-axis, like a topological suspension.}
    \end{figure}

\hspace{20pt} For definition of manifold with boundary (mwb), manifold with corners (mwc), boundary hypersurface (bhs), or (local) boundary-defining function (bdf), refer to Appendix \ref{ManifoldWithCorners}. The space $\overline{E}$ is a manifold with corners which is at most $2$-codimensional. The fiber above $(s,\theta)\in Z$ is an open $(n-1)$-ball tangent to the radius-$|s|$ sphere $S^{n-1}_{|s|}\subset\R^n$ at $s\theta$. This suggests the following identifications. Extend the embedding $\overline{Z}\hookrightarrow\R^1\times\R^n$ to an embedding of the tangent manifold $T\overline{Z}\hookrightarrow\R^1\times\R^n\times\R^1\times\R^n$ in the canonical way\footnote{The canonical extension of an embedding $\vphi:M\hookrightarrow\R^N$ is $\tilde\vphi:TM\ni v\mapsto(\pi(v),d\vphi_{\pi(v)}v)$.} and identify them. We have a smooth function $\I\mapsto T\overline{Z}$ defined by $(s,\theta,x)\mapsto(s,\theta,0,x-s\theta)$ which restricts to a bundle map $\vphi:\I\rightarrow\Ker(ds)\subset T\overline{Z}$. We only work on $\Ker(ds)$, for the remainder of the discussion, and will forget the third coordinate. Under this bundle map, the coordinate projection $P:\I\rightarrow\R\times S^{n-1}$ is identified with the canonical projection of $\Ker(ds)$. We equip $T\overline{Z}$ with the Riemannian metric $|\cdot|$ by inducing one from the standard inner product on each tangent space $\R^1\times\R^N$. The diffeomorphism $\vphi$ identifies the fiber of $E$ above $(s,\theta)$ with the open, radius-$\sqrt{1-s^2}$ tangent ball in the tangent space to $\{s\}\times S^{n-1}\subset \overline{Z}$. Define
    \begin{align}
        G = \Big\{\big((s,\theta),v\big)\in\Ker(ds) : |v|\leq1\Big\} = \Ker(ds)\cap D\overline{Z},
    \end{align}
as the closed unit ball bundle inside the tangent distribution $\Ker(ds)$, where $D\overline{Z}$ is the closed, tangent, unit ball bundle associated to the metric $|\cdot|$ on $T\overline{Z}$. For the remainder of the discussion we write $(s,\theta,v)$ for $\big((s,\theta),v\big)\in T\overline{Z}$. Define singular coordinates $\Upsilon:G\rightarrow \overline{E}$ as the rescaling of each tangent ball to the unit ball in the same fiber
    \begin{align}\label{defn:Upsilon}
        \Upsilon: (s,\theta,v)\mapsto \big(s,\theta,s\theta + \sqrt{1-s^2}\ v\big),
    \end{align}
within $\I$. Define the projections $\hat\pi$ and $\hat P$ via the commuting diagrams
	\begin{figure}[h]\centering
		\begin{tikzcd}
				& \arrow[dl, "\hat\pi" above] G \arrow[dr, "\hat P"] \arrow[d,"\Upsilon"]&\\
				\overline{\Omega} & \arrow[l,"\pi"] \overline{E} \arrow[r,"P" below] &\overline{Z}.
		\end{tikzcd}
	\end{figure}
    
Recall $\sigma(s,\theta) = 1-s^2$ is a local bdf on $\overline{Z}$.

\begin{prop}\label{lem:interior}
The restriction $\Upsilon|_{G^{\text{int}}}$ has $\Upsilon(G^{\text{int}}) = E$ and is a bundle map over $Z$. The induced volume element on $G$ by pullback is 
    \begin{align}\label{eq:pullbackvolume}
        \Upsilon^*(d\mu_{(s,\theta)}d\theta ds) = \sigma^{\frac{n-1}{2}}\ dv_{(s,\theta)}d\theta ds,
    \end{align}
where $dv_{(s,\theta)}$ is the volume element on the fiber $G|_{(s,\theta)}$. 
\end{prop}
\begin{proof}
The manifold $G^{\text{int}}$ is the restriction of $G$ to $|s| < 1$ and $|v| < 1$ and can be realized as the transverse intersection of smooth manifolds $\Ker(ds)\cap D^{\text{int}}Z$, where $D^{\text{int}}Z$ is the tangent, open ball bundle.  It is clear from definition \eqref{defn:Upsilon} that $\Upsilon$ is a fiber-wise mapping over $Z$ and is a diffeomorphism onto $E$ with inverse given by the inverse scaling of each fiber. Formula \eqref{eq:pullbackvolume} follows fiber-wise from the change of volume element of an open, $(n-1)$-ball by a uniform rescaling.
\end{proof}

Recall from Section \ref{SectionMainResults} that $\overline{\Omega}$ is a mwb and bdf $\rho(x) = 1-|x|^2$ and $\overline{Z}$ is a mwb with two components, each having bdf $(s,\theta)\mapsto 1\pm s$. Let $r$ denote $r:(s,\theta,v)\mapsto 1-|v|^2$ on $G$.

\begin{prop}\label{DoubleBFibrations}
$G$ is a manifold with corners with codimension at most $2$. The induced projections $\hat\pi:G\rightarrow\overline{\Omega}$ and $\hat P:G\rightarrow\overline{Z}$ are $b$-fibrations with geometric exponents 
	\begin{align}\label{eq:b-map}
		\hat\pi^*\rho = \sigma^1\ r^1, && \hat P^*(1\pm s) = (1\pm s)^1.
	\end{align}
$G$ is double $b$-fibered by $(\hat P,\hat\pi)$, in the sense of \cite{Melrose1992}. 
\end{prop}

\begin{proof}
By the local trivialization of the distribution $\Ker(ds)$, about each point $(s,\theta)\in\overline{Z}$ there is an open $U\subset\overline{Z}$ containing $(s,\theta)$ and $W\subset\R^{n-1}\times[0,\infty)$ such that 
    \begin{align}
        P^{-1}(U)\simeq W\times D^{n-1}
    \end{align}
where $D^{n-1}$ is the closed, unit $(n-1)$-ball. The trivialization $W\times D^{n-1}$ is a product of two mwb and gives $G$ the structure of a mwc of codimension at most $2$. A vector $(s,\theta,v)\in G$ is an interior point if both
    \begin{align}\label{eq:codim-cond}
        \sigma(s,\theta,v) = 1-s^2 > 0, && r(s,\theta,v) = 1-|v|^2 > 0,
    \end{align}
where in an abuse of notation, $\sigma$ denotes the pullback $\hat P^*\sigma$ to $G$. Together, $\sigma$ and $r$ are a complete set of (local) bdf for $G$. The codimension of the vector $(s,\theta,v)$ is determined by how many of \eqref{eq:codim-cond} fail. The manifold boundary of $G$ can be expressed as disjoint, equi-codimensional components  
	\begin{align}\label{eq:G-decomp}
		\partial G = (D^{\text{int}}\overline{Z}\cap G|_{\partial Z})\cup(S\overline{Z}\cap G|_{Z})\cup(S\overline{Z}\cap G|_{\partial Z}),
	\end{align}
where $D^{\text{int}}\overline{Z}$ is the open ball bundle and $S\overline{Z}$ is the sphere bundle corresponding to $D\overline{Z}$ 

\hspace{20pt} Formulas \eqref{eq:b-map} follows from computing the two pullbacks. For $\hat\pi^*\rho$, the identity $v\perp\theta$ for all $(s,\theta,v)\in\Ker(ds)$ is used. Together with Proposition \ref{lem:interior} give $\hat P$ and $\hat\pi$ the structure of $b$-maps in the sense of \cite{Melrose1992}. From \eqref{eq:b-map}, it is clear both $\hat P|_{G^{\text{int}}}$ and $\hat\pi|_{G^{\text{int}}}$ map into $Z$ and $\Omega$ respectively. It follows that both $\hat\pi$ and $\hat P$ are $b$-normal since each map to a mwb with $\coDim\leq1$ while each subface of $G$ has $\coDim\geq 1$ such that \eqref{b-normality} is immediate for the boundary faces. Moreover, the map $\hat P$ is the restriction of the canonical projection from $\Ker(ds)$ and further restriction to each component of \eqref{eq:G-decomp} defines an open, $(n-1)$-ball bundle over $\partial Z$ or an $(n-2)$-sphere bundle over $Z$ or $\partial Z$, respectively. This shows $\hat P$ is a $b$-fibration. 

\hspace{20pt} We determine $\hat\pi$ is a $b$-fibration by examining the restriction to each face. Restriction to $s^2=1$ (corresponding to $\hat\pi|_{G|_{\partial Z}}$) has factorization into fibrations $\hat\pi(\pm1,\theta,v) = \pm\theta = \pm p\circ\hat P(\pm1,\theta,v)$ where $p:\overline{Z}\mapsto S^{n-1}$ is a trivial bundle. It follows that $\hat\pi|_{G|_{\partial Z}}$ is either an open, $(n-1)$-ball bundle or an $(n-2)$-sphere bundle over $\partial\Omega$, up to the antipodal map of $\partial\Omega$, depending on whether $|v| < 1$ or $|v| = 1$. Now suppose $s^2 < 1$. The fiber above $x\in\overline{\Omega}$ is
    \begin{align}\label{eq:hat-pi-fiber}
        \hat\pi^{-1}\{x\} = \Big\{(s,\theta,v)\in G|_Z : x = s\theta + \sqrt{1-s^2}\ v\Big\}\simeq\begin{cases}
            S^{n-1}, & |x|^2 < 1,\\ S^{n-1}\setminus\{\pm x\}, & |x|^2 = 1,
        \end{cases}
    \end{align}
where the diffeomorphism arises from $p\circ\hat P:\hat\pi^{-1}\{x\}\rightarrow S^{n-1}$. The restriction $\hat\pi|_{G^{\text{int}}}$ corresponds to the first case in \eqref{eq:hat-pi-fiber}. The fiber trivialization \eqref{eq:hat-pi-fiber} extends to a bundle map $J:G^{\text{int}}\rightarrow\Omega\times S^{n-1}$ defined by
    \begin{align}
        J:(s,\theta,v)\mapsto\big(\hat\pi(s,\theta,v),\theta\big), && J^{-1}:(x,\theta)\mapsto\Big(x\cdot\theta,\theta,\frac{x-s\theta}{\sqrt{1-s^2}}\Big).
    \end{align}
The bundle map is well-defined as $(s,\theta,x)\in G^{\text{int}}$ implies $|s| = |x\cdot\theta|\leq|x| < 1$ by Cauchy-Schwarz. The second case of \eqref{eq:hat-pi-fiber} corresponds to $\hat\pi|_{S\overline{Z}\cap G|_Z}$. The formula defining $J$ also defines a local trivialization of the cylinder bundle $\hat\pi|_{S\overline{Z}\cap G|_Z}$ but fails to extend globally. Explicitly, in a small neighborhood of $x\in U\subset\partial\Omega$, there is a unique extension $A_{\underline{\;\;}}:U\rightarrow SO(n)$ of identity such that $A_yy=x$, $A_y(y^{\perp}) = x^{\perp}$, and $A_y$ fixes $\Span\{x,y\}^{\perp}$ for each $y\in U$. The local trivialization is $(\hat\pi,A_{\hat\pi}) : J(\hat\pi^{-1}(U))\rightarrow U\times(S^{n-1}\setminus\{\pm x\})$ where the second component is defined by rotation $A_{\hat\pi}\theta$. 
\end{proof}

\hspace{20pt} In a double fibration setting, we can formulate integral operators by the pullback and pushforward operators generated by the pair of fibrations. Formulas for the adjoint operators are given by the dual formulas (see \cite{Gelfand1969,Helgason1980}). In the double $b$-fibration setting, we can still formulate integral operators and adjoints in the pullback and pushforward operators generated by the pair of $b$-fibrations, up to the possible inclusion of singular operators in the formulas. For a ($b$-)fibration $F$, let $F_*$ denote the pushforward map between measures. 

\begin{prop}\label{Prop:Pullback-Pushforward}
For $u\in C^{\infty}_c(\Omega)$,
	\begin{align}\label{eq:RadonPullPush}
		Ru\ d\theta ds &= \sigma^{\frac{n-1}{2}}\ \hat P_*\big(\hat\pi^*u\ dv_{(s,\theta)}d\theta ds\big).
	\end{align}
For $w\in C^{\infty}_c(Z)$,
	\begin{align}\label{eq:BackProjPullPush}
		R^*w\ dx &= \hat\pi_*\big(\sigma^{\frac{n-1}{2}}\ \hat P^*w\ dv_{(s,\theta)}d\theta ds\big).
	\end{align}
\end{prop}
\begin{proof}
For $u\in C^{\infty}_c(\Omega)$, the Radon transform $Ru$ can be realized as the pullback-pushfoward formula $Ru\ d\theta ds =  P_*(\pi^*u\ d\mu_{(s,\theta)}d\theta ds)$ where $P$ and $\pi$ are restricted to $E$ \eqref{eq:incidence-rel} (see \cite{Gelfand1969,Helgason1980}). By definition, $\supp(\pi^*u)\subset E$ and by Proposition \ref{lem:interior}, $\Upsilon|_{G^{\text{int}}}$ \eqref{defn:Upsilon} is a diffeomorphism on $\supp(\pi^*u)\subset E$ and acts as the change of coordinates by 
    \begin{align}\label{eq:change-of-coor}
		Ru\ d\theta ds &= \int_{P^{-1}(s,\theta)}\pi^*u\ d\mu_{(s,\theta)}d\theta ds \\
        &= \int_{\Upsilon^{-1}\big(P^{-1}(s,\theta)\big)}\hat\pi^*u\ \Upsilon^*(d\mu_{(s,\theta)}d\theta ds)\\
													&= \int_{\hat P^{-1}(s,\theta)}\hat\pi^*u\ \sigma^{\frac{n-1}{2}}dv_{(s,\theta)}d\theta ds,\\
													&= \sigma^{\frac{n-1}{2}}\int_{\hat P^{-1}(s,\theta)}\hat\pi^*u\ dv_{(s,\theta)}d\theta ds = \sigma^{\frac{n-1}{2}}\hat P_*(\hat\pi^* u\ dv_{(s,\theta)}d\theta ds),
	\end{align}
where the pullback volume \eqref{eq:pullbackvolume} was used. The pushforward $\hat P_*$ commutes with multiplication by $\sigma^{\frac{n-1}{2}}$ due to the fibers of $\hat P$ being independent of $s$. Formula \eqref{eq:BackProjPullPush} follows by duality.
\end{proof}

Before discussing the proof of Theorem \ref{RadonExpansion}, we prove Lemma \ref{IntegralFiberSphere} which states the phg structure of a function appearing in the proof of Theorem \ref{RadonExpansion} as well as Theorem \ref{MainResult}. This function arises in studying the integration of a phg component over a fiber of $G$. See Appendix \ref{AlgebraOnManifold} for the meaning of phg expansion. For $a\in C^{\infty}(S^{n-1})$, let $\tilde a(x) = a(\frac{x}{|x|})$ denote the radial extension of $a$ to smooth function in a neighborhood of $S^{n-1}\subset\overline{\Omega}$.

\begin{lem}\label{IntegralFiberSphere}
Let $a\in C^{\infty}(S^{n-1})$ and $\chi\in C^{\infty}_c[0,1)$ a cutoff $\chi(\rho)\equiv1$ near $0$. The function $A:\overline{Z}\times[-1,1]_{\tau}\rightarrow\C$ defined by
	\begin{align}\label{eq:fiber-integral}
		A(s,\theta,\tau) = \int_{\partial\Omega\cap\theta^{\perp}}\hat\pi_{(s,\theta)}^*(\tilde a\chi)(\tau\phi)\ d\phi_{\theta},
	\end{align}
where $d\phi_{\theta}$ is the Lebesgue measure on the $(n-2)$-sphere $\partial\Omega\cap\theta^{\perp}$, is smooth and symmetric with respect to the involutions $(s,\theta)\mapsto(-s,-\theta)$ and $\tau\mapsto-\tau$. There exists $A_p\in C^{\infty}(S^{n-1})$ such that near $\sigma = 0$, $A$ has phg expansion 
	\begin{align}\label{eq:fiber-integral-expansion}
		A(s,\theta,\tau)\sim\sum_{m\geq0}\sigma^m\sum_{p=0}^mA_p\big(\textstyle\frac{s}{|s|}\theta\big)\ \displaystyle{m-1\choose p-1}\tau^{2p}.
	\end{align}
\end{lem}
\begin{proof}
The symmetry of \eqref{eq:fiber-integral} under $(s,\theta)\mapsto(-s,-\theta)$ follows from $\hat\pi_{(s,\theta)} = \hat\pi_{(-s,-\theta)}$ and the identification of the spaces $(-\theta)^{\perp} = \theta^{\perp}$ and measures $d\phi_{\theta}=d\phi_{-\theta}$. The symmetry in $\tau\mapsto-\tau$ follows from applying the $d\phi_{\theta}$-preserving antipodal reparameterization to $\phi\in\partial\Omega\cap\theta^{\perp}$ and to replace $\tau$ with $-\tau$. We use the evenness in $\tau$ to assume $\tau\in[0,1]$ for the remainder of the proof. 

\hspace{20pt} Evaluating the pullback in \eqref{eq:fiber-integral}, 
    \begin{align}\label{eq:pullback-eval}
        \int_{\partial\Omega\cap\theta^{\perp}}\hat\pi_{(s,\theta)}^*(\tilde a\chi)(\tau\phi)\ d\phi_{\theta} = \chi((1-\tau^2)\sigma)\underbrace{\int_{\partial\Omega\cap\theta^{\perp}}a\bigg(\frac{s\theta + \sigma^{1/2}\tau\phi}{\sqrt{s^2 + \sigma \tau^2}}\bigg)\ d\phi_{\theta}}_{J(\theta;s,\tau)}.
    \end{align}
Fix $(s,\tau)\in[-1,1]\times[0,1]$. We show \eqref{eq:fiber-integral} is smooth on $S^{n-1}$. For fixed $\theta$ as reference, let $\theta\in U\subset\partial\Omega$ be an open neighborhood small enough that there is a unique, smooth extension $A_{-}:U\rightarrow SO(n)$ of the identity $A_{\theta} = id$ such that for all $\theta'\in U$ we have $A_{\theta'}\theta' =  \theta$, $A_{\theta'}({\theta'}^{\perp})= \theta^{\perp}$, and $A_{\theta'}$ fixes $\Span(\theta,\theta')^{\perp}$, then for $\theta'\in U$
    \begin{align}
        J(\theta';s,\tau) &= \int_{A_{\theta'}(\partial\Omega\cap{\theta'}^{\perp})}A_{\theta'}^{-*}a\bigg(\frac{s A_{\theta'}\theta' + \sigma^{1/2}\tau A_{\theta'}\phi'}{\sqrt{s^2 + \sigma\tau^2}}\bigg)|-1|^{n-2}d\phi_{\theta}\\
        &= \int_{\partial\Omega\cap\theta^{\perp}}A_{\theta'}^{-*}a\bigg(\frac{s\theta + \sigma^{1/2}\tau\phi}{\sqrt{s^2 + \sigma\tau^2}}\bigg)d\phi_{\theta},
    \end{align}
is smooth in $U$ since $A_{\theta'}$ and $\phi\mapsto a(\frac{s\theta + \sigma^{1/2}\tau\phi}{\sqrt{s^2 + \sigma\tau^2}})$ are smooth on $\supp(\hat\pi_{(s,\theta)}^*\chi)$ and $\partial\Omega\cap\theta^{\perp}$ is compact. 

\hspace{20pt} Now fix $\theta\in S^{n-1}$ and consider $(s,\tau)\in(-1,1)\times[0,1]$. The assignment $(s,\tau)\mapsto a(\frac{s\theta + \sigma{1/2}\tau\phi}{\sqrt{s^2 + \sigma\tau^2}})$ is smooth on $\supp(\hat\pi_{(s,\theta)}^*\chi)$ (recall $\sigma^{1/2}=\sqrt{1-s^2}$ is smooth on $(-1,1)\ni s$) and $\partial\Omega\cap\theta^{\perp}$ compact shows \eqref{eq:fiber-integral} is smooth near $s=0$ and in $\tau$ when $s\in(-1,1)$.

\hspace{20pt} Now keep $\theta\in S^{n-1}$ fixed and initially fix $0 < s < 1$. From \eqref{eq:pullback-eval}, write 
	\begin{align}\label{eq:defn-A}
		\hat\pi_{(s,\theta)}^*\tilde a\ (\tau\phi) = a\bigg(\frac{s}{|s|}\ \frac{\theta + \frac{\sigma^{1/2}\tau}{s}\phi}{\sqrt{1 + (\frac{\sigma^{1/2}\tau}{s})^2}}\bigg) = a\bigg(\frac{\theta + \frac{\sigma^{1/2}\tau}{s}\phi}{(1 + (\frac{\sigma^{1/2}\tau}{s})^2)^{1/2}}\bigg).
	\end{align}
Set $q = \frac{\sigma^{1/2}\tau}{s}$ and consider $F(-;\theta):\R\rightarrow\C$ defined by 
    \begin{align}
        F(q;\theta) = \int_{\partial\Omega\cap\theta^{\perp}}a\bigg(\frac{\theta + q\phi}{(1 + q^2)^{1/2}}\bigg)d\phi_{\theta}.
    \end{align}
The smoothness of $a$ and of $q\mapsto (\theta + q\phi)/\sqrt{1 + q^2}$ from $\R\rightarrow S^{n-1}$ and the compactness of $\partial\Omega\cap\theta^{\perp}$ ensures $F$ is smooth in $q$. Moreover, if we reparameterize the integration by $\phi = -\phi'$, which has Jacobian $|-1|^{n-2} =1$, we recover the same integral with $q$ replaced with $-q$ and shows that $F(q;\theta)$ is even in $q$. By evenness in $q$ and smoothness in both variables, the Taylor expansion of $F(q;\theta)$ at $q=0$ is in even powers of $q$ and has coefficients $A_p\in C^{\infty}(S^{n-1})$. Substitute $q = \frac{\sigma^{1/2}\tau}{s}$ for $q$ to find
    \begin{align}
        F(\textstyle\frac{\sigma^{1/2}\tau}{s};\theta) \sim \displaystyle\sum_{p=0}^{\infty}A_p(\theta)\ \sigma^p\ \tau^{2p}\ s^{-2p}.
    \end{align}
in the sense of \eqref{asymptoticequivalence}. Expand $s^{-2p} = (1-\sigma)^{-p} = \sum_{m=0}^{\infty}{m+p-1\choose p-1}\sigma^m$ valid for $|\sigma|<1$ (or $0 < s^2\leq1$) and convolve the coefficients to find \eqref{eq:fiber-integral-expansion} for $s > 0$. This expansion also shows \eqref{eq:fiber-integral} is smooth in $(s,\tau)$ near $s=1$ as all components are smooth in $(s,\tau)$.

\hspace{20pt} For $-s > 0$, we again use $\hat\pi_{(s,\theta)} =\hat\pi_{(-s,-\theta)}$ and $\theta^{\perp} = (-\theta)^{\perp}$ and $d\phi_{\theta}=d\phi_{-\theta}$ to write   
    \begin{align}
        \int_{\partial\Omega\cap\theta^{\perp}}\hat\pi_{(s,\theta)}^*(\tilde a\chi)(\tau\phi)\ d\phi_{\theta} = \int_{\partial\Omega\cap(-\theta)^{\perp}}\hat\pi_{(-s,-\theta)}^*(\tilde a\chi)(\tau\phi)\ d\phi_{-\theta}.
    \end{align}
Expansion \eqref{eq:fiber-integral-expansion} is valid for $-s > 0$ at the cost of replacing $\theta$ with  $-\theta$ and shows \eqref{eq:fiber-integral-expansion} is smooth in $(s,\tau)$ near $s = -1$.
\end{proof}

\subsection{Proof of Theorem \ref{RadonExpansion}}\label{SubsectionProofRadon}

\begin{proof}[Proof of Theorem \ref{RadonExpansion}.]
Fix $(s,\theta)\in Z$ away from $s = 0$ and write $Ru(s,\theta) = R_{(s,\theta)}u$. Set $x =  \hat\pi(s,\theta,v) = \hat\pi_{(s,\theta)}(v)$. From \eqref{eq:RadonPullPush} and $\hat P^{-1}(s,\theta) = \{(s,\theta)\}\times(\overline{\Omega}\cap\theta^{\perp})$ with Lebesgue measure $dv_{(s,\theta)} = dv_{\theta}$ (the topology of the fiber is independent of $s$), write 
	\begin{align}
		R_{(s,\theta)}u = \int_{\overline{\Omega}\cap\theta^{\perp}}\hat\pi_{(s,\theta)}^*u\ \sigma^{\frac{n-1}{2}}dv_{\theta},
	\end{align}
where we have identified $R_{(s,\theta)}u$ with the density $R_{(s,\theta)}u\ d\theta ds$. Set $u = \tilde a(x)\ \rho^{\gamma}\log(\rho)^{\el}\ \chi(\rho)$ as in the statement of Theorem \ref{RadonExpansion}. From \eqref{eq:b-map}, $\hat\pi_{(s,\theta)}^*\rho = r\sigma$. The logarithm $\hat\pi_{(s,\theta)}^*\log(\rho)^{\el}$ has binomial expansion 
	\begin{align}\label{RadonPullbackFormula}
		R_{(s,\theta)}u &= \sigma^{\frac{n-1}{2} + \gamma}\sum_{k=0}^{\el}\log(\sigma)^{\el}\ {k\choose \el}\int_{\overline{\Omega}\cap\theta^{\perp}}\hat\pi_{(s,\theta)}^*(\tilde a\chi)\ r^{\gamma}\ \log(r)^{\el-k}\ dv_{\theta}\\
        &=\sigma^{\frac{n-1}{2} + \gamma}\sum_{k=0}^{\el}\log(\sigma)^l{\el\choose k}\ \partial_{\eta}^{\el-k}\Big|_{\eta = \gamma}\int_{\overline{\Omega}\cap\theta^{\perp}}\hat\pi_{(s,\theta)}(\tilde a\chi)\ r^{\eta}\ dv_{\theta},
	\end{align}
where the differentiation in $\eta$ is valid due absolute integrability (the integrand is bounded and $\overline{\Omega}\cap\theta^{\perp}$ compact) on compact subsets in the region $\Real(\eta) > -1$. It remains to study 
    \begin{align}\label{eq:FiberIntegralRadon}
        I_{(s,\theta)}(\eta) = \int_{\overline{\Omega}\cap\theta^{\perp}}\hat\pi_{(s,\theta)}^*(\tilde a\chi)\ r^{\eta}\ dv_{\theta}.
    \end{align}
Set $v = \tau\phi$ where $\tau = |v|$ and $\phi\in\partial\Omega\cap\theta^{\perp}$, expand $r = 1-\tau^2$ and rearrange \eqref{eq:FiberIntegralRadon} to write 
	\begin{align}\label{eq:FiberIntegralRadonChangeVariable}
		I_{(s,\theta)}(\eta) = \int_{\tau = 0 }^1(1-\tau^2)^{\eta}\ \tau^{n-2}\ \int_{\partial\Omega\cap\theta^{\perp}}\hat\pi_{(s,\theta)}^*(\tilde a\chi)(\tau\phi)\ d\phi_{\theta}\ d\tau,
	\end{align}
where the inner integral is \eqref{eq:fiber-integral}. By Lemma \ref{IntegralFiberSphere}, we may substitute expansion \eqref{eq:fiber-integral-expansion} in \eqref{eq:FiberIntegralRadonChangeVariable}, exchange integration and asymptotic summation, and expand $\chi((1-\tau^2)\sigma)\sim1$ near $\sigma=0$ to find 
	\begin{align}\label{eq:FiberIntegralRadonExpansion}
		I_{(s,\theta)}(\eta)\sim\frac{1}{2}\sum_{m\geq0}\sigma^m\sum_{p=0}^mA_p\big(\textstyle\frac{s}{|s|}\theta\big)\displaystyle{m-1\choose p-1}\ \int_0^1(1-\tau^2)^{\eta}\ \tau^{n-3+2p}\ 2\tau\ d\tau.
	\end{align}
We recognize the integral in $\tau$ as $B(\eta + 1,\frac{n-1}{2} + p)$ \eqref{eq:beta} after changing variables to $u = \tau^2$. Substituting \eqref{eq:FiberIntegralRadonExpansion} to equation \eqref{RadonPullbackFormula} and rearranging finds \eqref{eq:RadonExpansion} with coefficients \eqref{eq:Radon-coefficients}. Integral formula \eqref{eq:beta} for $B(\eta + 1,\beta)$ requires $\Real(\eta + 1) > 0$, showing the necessity of the integrability condition on $\gamma$. 
\end{proof}

\subsection{Proof of Theorem \ref{MainResult}}\label{SubsectionProofBackproj}

\textbf{Setting for proof.} We study $R^*u$ for $u$ a phg component on $\overline{Z}$ by identifying $R^*u$ with the $b$-density $\mu = R^*u\frac{dx}{\rho}$ (see Appendix \ref{PullbackPushforwardTheorems} for definition of a phg $b$-density). We associate to $\mu$ the meromorphic function $\mu_M$, valued in densities $|\Lambda^{n-1}(\partial\Omega)|$\footnote{The manifold $\partial\Omega$ is closed. The $b$-densities $\prescript{b}{}\Omega(\partial\Omega)$ and the smooth densities $|\Lambda^{n-1}(\partial\Omega)|$ coincide.} and given by a Mellin transform 
	\begin{align}\label{eq:Mellin-mu}
		\mu_M(z) = \int_{\rho=0}^{\infty}\chi(\rho)\rho^z\ \mu,
	\end{align}
where $\chi\in C^{\infty}_c[0,1)$ is any fixed cutoff with $\chi\equiv1$ near $0$. This defines a correspondence between the phg components of $R^*u$ and the poles of $\mu_M$ in the Mellin variable $z$ (see also Appendix \ref{MellinBackground}). To study the meromorphic, density-valued $\mu_M$, we fix $h\in C^{\infty}(\partial\Omega)$ and associate the meromorphic function $\langle h,\mu_M\rangle_{\partial\Omega}$ defined by the pairing $C^{\infty}(\partial\Omega)\times|\Lambda^{n-1}(\partial\Omega)|\rightarrow\C$ of smooth functions and densities given by integration of $h\mu$. 

\hspace{20pt} We demonstrate a few Propositions and Lemmas before arguing for Theorem \ref{MainResult}.  Recall $\tilde h(x) = h(\frac{x}{|x|})$ denotes the radial extension of $h\in C^{\infty}(\partial\Omega)$ to a smooth function in a neighborhood of $\partial\Omega\subseteq\overline{\Omega}$. 

\begin{prop}\label{MellinPairing}
Let $a\in C^{\infty}(S^{n-1})$, $(\gamma,\el)\in\C\times\N_0$ and $\chi\in C^{\infty}_c[0,1)$ a cutoff $\chi\equiv1$ near $0$. Let $z\in\C$ with $\Real(z) > \max\{0,-\frac{n-1}{2} - \Real(\gamma)\}$. Set $u = a(\theta)\ \sigma^{\gamma}\log(\sigma)^{\el}\ \chi(1\mp s)$ and $\mu = R^*u\frac{dx}{\rho}$. Let $h\in C^{\infty}(\partial\Omega)$, then the pairing  
	\begin{align}\label{eq:pairing-prefinal}
		\langle h,\mu_M\rangle_{\partial\Omega}(z;\gamma,\el) = \int_{-1}^1\sigma^{\frac{n-1}{2} + z}\ \sigma^{\gamma}\log(\sigma)^{\el}\chi(1\mp s)\int_{S^{n-1}}a(\theta)\  B(s,\theta;z)\frac{d\theta ds}{\sigma},
	\end{align}
where the mapping $\mu\mapsto\mu_M$ uses the same cutoff $\chi$ as $u$ and 
	\begin{align}
		B(s,\theta;z) = B_h(s,\theta;z)  = \int_{\overline{\Omega}\cap\theta^{\perp}} \hat\pi^*(\tilde h\chi)(s,\theta,v)\ r^{z-1}\ dv_{(s,\theta)}.\label{B-integral-formula}
	\end{align}
\end{prop}
\begin{proof}
We can write 
    \begin{align}\label{eq:pairing-1}
        \langle h,\mu_M\rangle_{\partial\Omega} &= \int_{\partial\Omega}h(\omega)\int_{\rho = 0}^{\infty}\chi(\rho)\rho^z\mu =\int_{\overline{\Omega}}\tilde h(x)\chi(\rho)\rho^z\mu = \int_{\overline{\Omega}}\tilde h(x)\chi(\rho)\rho^{z-1}R^*u\ dx,
    \end{align}
by replacing the region $(0,\infty)_{\rho}$ with $(0,1)_{\rho}$ by $\supp(\chi)\subset[0,1]_{\rho}$ and realizing the double integral as the single integral over $\overline{\Omega}$ by the Fubini-Tonelli theorem (the use of Fubini-Tonelli will be justified later by our choice of $z$). This is well-defined despite $\tilde h$ being singular at $x=0$ as $\supp(\chi)$ defines a neighborhood away from $\rho=0$ when viewed as a function on $\chi:\overline{\Omega}\rightarrow\C$. Substitute the pullback-pushforward formula \eqref{eq:BackProjPullPush} for $R^*u$ from Proposition \ref{Prop:Pullback-Pushforward} in \eqref{eq:pairing-1} and pushforward-pullback duality to write
	\begin{align}\label{eq:pairing-2}
		\int_{\overline{\Omega}}\tilde h(x)\chi(\rho)\rho^{z-1}\ \hat\pi_*\Big(\sigma^{\frac{n-1}{2}}\ \hat P^*u\ dv_{\theta}d\theta ds\Big) &= \int_G\hat\pi^*(\tilde h \chi \rho^{z-1})\sigma^{\frac{n-1}{2}}\hat P^*u\ dv_{\theta}d\theta ds.
	\end{align}
The pullback $\hat\pi^*(\rho^{z-1}) = \frac{\sigma^{z}}{\sigma}r^{z-1}$ following \eqref{eq:b-map} from Proposition \ref{DoubleBFibrations}. The pullback $\hat P^*u$ keeps the same formula as $u$ viewing $s\in[-1,1]$ as a variable on $G$ rather than $\overline{Z}$. Write \eqref{eq:pairing-2} as the double integral over the fiber, then the base $\overline{Z}$ to find  
    \begin{align}\label{eq:pairing-3}
		\langle h,\mu_M\rangle_{\partial\Omega} = \int_{\overline{Z}} \sigma^{\frac{n-1}{2} + z+\gamma}\log(\sigma)^{\el}\ \chi(1\pm s)\ a(\theta)\ B(s,\theta;z)\frac{dsd\theta}{\sigma},
	\end{align}
where $B(s,\theta;z)$ is given by formula \eqref{B-integral-formula}. 

\hspace{35pt } Use the boundedness of $\hat\pi^*(\tilde h\chi)$ on $\overline{\Omega}$ and polar coordinates $v=\tau\phi$ on $\overline{\Omega}\cap\theta^{\perp}$ to write 
    \begin{align}\label{eq:B-bound}
        |B|\leq C\int_0^1(1-\tau^2)^{\Real(z)-1}\tau^{n-2}d\tau.
    \end{align}
After setting $u = \tau^2$, the upper bound is a multiple of the Beta function \eqref{eq:beta} with $\alpha=\Real(z)$ and $\beta=\frac{n-1}{2} > 0$ and is uniformly bounded on compact subsets of $\{z\in\C : \Real(z) > 0\}$. Fix a compact subset of $\{z\in\C : \Real(z) > 0\}$ and bound \eqref{eq:pairing-3} by 
    \begin{align}
        |\langle h,\mu_M\rangle_{\partial\Omega}|\leq C\int_{-1}^1 \sigma^{\Real(\frac{n-1}{2} + z+\gamma)}\log(\sigma)^{\el}\frac{ds}{\sigma}\leq C'\int_{0}^1 \sigma^{\Real(\frac{n-1}{2} + z+\gamma-1)}\ (1-\sigma)^{\frac{1}{2}-1}\log(\sigma)^{\el}d\sigma,
    \end{align}
which is bounded when $\Real(z) > -\frac{n-1}{2}-\Real(\gamma)$ by an $\el$th derivative of \eqref{eq:beta} in the $\gamma$ variable, justifying the use of Fubini-Tonelli for $z\in\C$ with $\Real(z) > \max\{0,-\frac{n-1}{2}-\Real(\gamma)\}$ (constants $C$, and $C'$ maybe negative depending on the parity of $\el$).
\end{proof}

\hspace{20pt} The next considerations are made of $B(s,\theta;z)$ and its functional properties. We first restate \cite[Proposition 5.1]{Mazzeo2021} about the meromorphic extension of the beta functional for completeness

\begin{prop*}[5.1, \cite{Mazzeo2021}]
If $f\in C^k[0,1]$ is fixed by $t\mapsto(1-t)$, then the beta functional  
    \begin{align}\label{eq:beta-func}
        \beta[f](z) = \int_0^1f(t)\ t^{z-1}(1-t)^{z-1}dt,
    \end{align}
satisfies the functional relation 
    \begin{align}\label{eq:beta-func-rel}
        \beta[f](z) = \frac{2}{z}\Big((2z+1)\beta[f](z+1) + \beta\big[(\cdot-\textstyle\frac{1}{2})f'\big](z+1)\Big),
    \end{align}
and has meromorphic extension to $\{z\in\C:\Real(z) > -k\}$ with at most simple poles at $0,-1,...,-k+1$.
\end{prop*}

In the proof of \cite[Proposition 5.1]{Mazzeo2021}, the algebraic identity $(u-1/2)^2 = 1/4 - u(1-u)$ was used. Using the binomial theorem, for $\el\in\N$ we also have the identity
	\begin{align}\label{interval-identity}
		(2t-1)^{2\el} = 4^{\el}\cdot(t-1/2)^{2\el} = 4^{\el}\big(4^{-1} - (1-t)t\big)^{\el} = \sum_{q = 0}^{\el}{\el\choose q}(-4)^{q}\big((1-t)t\big)^q.
	\end{align}

\hspace{20pt} We also require a version of \cite[Proposition 5.1]{Mazzeo2021} for a Mellin functional on $[0,1]$, defined for $f\in C[0,1]$ by \eqref{eq:mellin-func}. We reprove the result for completeness.

\begin{prop}\label{prop-mellin}
If $f\in C^k[0,1]$, then 
    \begin{align}\label{eq:mellin-func}
        M[f](z) = \int_0^1f(\xi)\ \xi^{z-1}d\xi,
    \end{align}
has meromorphic extension to $\{z\in\C:\Real(z) > -k\}$ with at most simple poles at $0,-1,...,-k+1$.
\end{prop}
\begin{proof}
Let $z$ have $\Real(z) > 0$. Replace $f$ with its Taylor expansion at $\xi=0$ and rearrange for
    \begin{align}
         M[f](z) = \sum_{q=0}^{k-1}\frac{f^{(q)}(0)}{q!}\int_0^1\xi^{q+z-1}d\xi + M[R_k](z),
    \end{align}
where $R_k\in\A^k[0,1]$ is conormal with respect to $\xi=0$. For each $q$, the integral is the simple pole $\frac{1}{z+q}$ at $z=-q$ which extend to all of $\C$, while the cornomality of the remainder $R_k$ ensures $M[R_k](z)$ is analytic for $z$ with $\Real(z + k) > 0$. 
\end{proof}

\begin{prop}\label{BProperties}
Let $B(s,\theta;z)$ be defined by \eqref{B-integral-formula}. $B$ is smooth on $\overline{Z}$ and analytic on $\{z\in\C : \Real(z) > 0\}$ and has meromorphic extension to all of $\C$ with poles at most $-\N_0$ that are at most simple. For every $z$ away from $-\N_0$, $B(s,\theta;z) = B(-s,-\theta;z)$.
\end{prop}
\begin{proof}
Formula \eqref{B-integral-formula} defines an analytic function on $\{z\in\C : \Real(z) > 0\}$ since the bound \eqref{eq:B-bound} is uniform on any of compact subset belonging to it. In polar coordinates $v = \tau\phi$,
    \begin{align}\label{eq:B-new-coord}
		B(s,\theta;z)  = \int_{\tau=0}^1(1-\tau^2)^{z-1}\ \tau^{n-2}\int_{\partial\Omega\cap\theta^{\perp}}\hat\pi_{(s,\theta)}^*(\tilde h\chi)(\tau\phi) \ d\phi_{\theta} d\tau.
	\end{align}
The inner integral is \eqref{eq:fiber-integral} with $a = h\in C^{\infty}(\partial\Omega)$. Lemma \ref{IntegralFiberSphere} states \eqref{eq:fiber-integral} is smooth on $\overline{Z}$ and even with respect to $(s,\theta)\mapsto(-s,-\theta)$, properties $B(s,\theta;z)$ adopts for $z$ with $\Real(z) > 0$. 

\hspace{20pt} Set $B(s,\theta;z) = B_{(s,\theta)}(z)$. Smoothness and evenness of \eqref{eq:fiber-integral} in $\tau$ implies there exists $H_{(s,\theta)}\in C^{\infty}[0,1]$ satisfying 
    \begin{align}\label{eq:factored-fiber-integral}
        H_{(s,\theta)}(1-\tau^2) = \int_{\partial\Omega\cap\theta^{\perp}}\hat\pi_{(s,\theta)}^*(\tilde h\chi)(\tau\phi) d\phi_{\theta},
    \end{align}
which remains smooth in $(s,\theta)$ and fixed by $(s,\theta)\mapsto(-s,-\theta)$. Let $n$ be even. Substitute \eqref{eq:factored-fiber-integral} into \eqref{eq:B-new-coord} and set $\tau = 1-2t$ to find
	\begin{align}\label{eq:B-new-coord-2}
		B_{(s,\theta)}(z) &= 4^{z-1}\ 2\int_{0}^{1/2}H_{(s,\theta)}\big(4(1-t)t\big)\ (1-t)^{z-1}t^{z-1}(1-2t)^{n-2}\ dt.
	\end{align}
By the evenness of $n$, $n-2$ is even and the integrand of \eqref{eq:B-new-coord-2} is fixed by $t\mapsto(1-t)$ while carrying $[0,\frac{1}{2}]$ to $[\frac{1}{2},1]$. We may replace the bounds with those of the invariant interval $[0,1] = [0,\frac{1}{2}]\cup[\frac{1}{2},1]$ of the same transformation at the cost of $\frac{1}{2}$. Again by the evenness of $n$, $\frac{n-2}{2}$ is integral and we may expand $(1-2t)^{n-2} = (1-2t)^{2(\frac{n-2}{2})}$ according to \eqref{interval-identity} to find 
	\begin{align}\label{eq:B-even-linear-sum}
		B_{(s,\theta)}(z) &= \sum_{q=0}^{\frac{n-2}{2}}{\frac{n-2}{2}\choose q}(-1)^q4^{z-1 + q}\ \beta\big[H_{(s,\theta)}\big(4(1-\cdot)\cdot\big)\big](z+q).
	\end{align}
The cited \cite[Proposition 5.1]{Mazzeo2021} above describes the meromorphic extension of $\beta\big[H_{(s,\theta)}\big(4(1-\cdot)\cdot\big)\big](w)$ for each $q$ to $\C$ as having at most simple poles at $w\in-\N_0$ (or $z\in-q-\N_0$) and the functional relation implies each term of \eqref{eq:B-even-linear-sum} adopts the smoothness on $\overline{Z}$ and symmetry under $(s,\theta)\mapsto(-s,-\theta)$. The poles of $B_{(s,\theta)}(z)$ are contained in the union of the support of the poles of each term of \eqref{eq:B-even-linear-sum}, $-\N_0$.

\hspace{20pt} Now let $n$ be odd, then $n-3$ is even and we expand $(2t-1)^{2(\frac{n-3}{2})} = (1-2t)^{2(\frac{n-3}{2})}$ using \eqref{interval-identity} and insert it into formula \eqref{eq:B-new-coord-2} to find
	\begin{align}\label{eq:B-new-coord-3}
		B_{(s,\theta)}(z) &= 2\sum_{q=0}^{\frac{n-3}{2}}{\frac{n-3}{2}\choose q}(-1)^q\int_{0}^{1/2}H_{(s,\theta)}\big(4(1-t)t\big)\big(4(1-t)t\big)^{z-1 + q}(1-2t)dt.
	\end{align}	
Set $\xi = 4(1-t)t$ in \eqref{eq:B-new-coord-3} to find 
	\begin{align}\label{eq:B-linear-sum}
		B_{(s,\theta)}(z) = \frac{1}{2}\sum_{q=0}^{\frac{n-3}{2}}{\frac{n-3}{2}\choose q}(-1)^q\ M\big[H_{(s,\theta)}\big](z+q),
	\end{align}	
where $M[\cdot]$ is the Mellin functional on $[0,1]$. By Proposition \ref{prop-mellin}, for each $q$, the meromorphic extension of $M\big[H_{(s,\theta)}\big](w)$ to $\C$ has at most simple poles at $w\in-\N_0$ (or $z\in-q-\N_0$). By the same Proposition, the Taylor coefficients of $H_{(s,\theta)}(\xi)$ at $\xi=0$ being smooth on $\overline{Z}$ and even with respect to $(s,\theta)\mapsto(-s,-\theta)$ ensures the extension of $M\big[H_{(s,\theta)}\big](w)$ adopts the same properties whenever defined. As before, the poles of $B_{(s,\theta)}(z)$ are contained in the union of the support of the poles of each term of \eqref{eq:B-linear-sum}, $-\N_0$.
\end{proof}

\hspace{20pt} Proposition \ref{BProperties} shows the extent of the pole structure of the meromorphic family $B(s,\theta;z)$. To understand how the pole structure of $B$ influences the meromorphic structure $\langle h,(R^*g)_M\rangle_{\partial\Omega}$ we must find the phg asymptotic expansion of $B$ near $\sigma=0$.

\begin{lem}\label{BDensityExpansion}
Following the hypothesis of Proposition \ref{MellinPairing}, the $b$-density $B(s,\theta;z)\frac{2d\theta ds}{\sigma}$ has, away from $z\in-\N_0$, polyhomogeneous expansion near $\sigma=0$,
	\begin{align}\label{eq:B-measure-expansion}
		B(s,\theta;z)\frac{2d\theta ds}{\sigma}\sim\sum_{m = 0}^{\infty}\sigma^m\ \frac{s}{|s|}\ \textstyle B_m(\frac{s}{|s|}\theta;z)\displaystyle\frac{d\theta d\sigma}{\sigma},
	\end{align}
where $B_m$ are meromorphic for each $\theta$ with at most simple poles at
	\begin{align}\label{eq:poles-B}
		\Lambda_{n,m} = \begin{cases}- \N_0, & n\text{ even,}\\ \{0,-1,...,-m - \frac{n-3}{2}\}, & n\text{ odd,}\end{cases}
	\end{align}
and smooth on $S^{n-1}$ whenever $z\notin\Lambda_{n,m}$. If $n$ is even, $B_m$ has simple zeros at $z\in-\frac{1}{2}- \frac{n-2}{2} - m-\N_0$. The most singular term of \eqref{eq:B-measure-expansion} is 
    \begin{align}\label{eq:B_0-most-sing}
        B_0(\textstyle\frac{s}{|s|}\theta;z) = h(\frac{s}{|s|}\theta)\displaystyle\ \frac{\omega_{n-2}}{2}\ \Gamma\Big(\frac{n-1}{2}\Big)\ \begin{cases}
            \displaystyle\frac{\Gamma(z)}{\Gamma(z + \frac{n-1}{2})}, & n\text{ even},\\
            \displaystyle\prod_{k=0}^{\frac{n-3}{2}}\frac{1}{z+k}, & n\text{ odd},
        \end{cases}
    \end{align}
where $\omega_{n-2}$ is the standard volume of the unit sphere $S^{n-2}$. 
\end{lem}
\begin{proof}
In polar coordinates $v = \tau\phi$, \eqref{B-integral-formula} becomes \eqref{eq:B-new-coord}. The inner integral is \eqref{eq:fiber-integral} with $h\in C^{\infty}(\partial\Omega)$ playing the role of $a$ and is the subject of Lemma \ref{IntegralFiberSphere}. Fix a region $\pm s > 0$. We substitute expansion \eqref{eq:fiber-integral-expansion} to \eqref{eq:B-new-coord} and exchange the order of summation and integration to  find 
    \begin{align}\label{eq:B-expansion}
		B\sim\sum_{m = 0}^{\infty}\sigma^m\ \sum_{p=0}^mH_p(\textstyle\frac{s}{|s|}\displaystyle\theta){m-1\choose p-1}\int_0^1(1-\tau^2)^{z-1}\ \tau^{n-2 + 2p}d\tau,
	\end{align}
where $H_p$ are used instead of $A_p$. For each $p\in\N_0$, setting $u = \tau^2$ in the integral in \eqref{eq:B-expansion} reveals a Beta function \eqref{eq:beta}
    \begin{align}\label{eq:beta-func-sum-1.1}
        \frac{1}{2}\int_0^1(1-u)^{z-1}\ u^{\frac{n-3}{2} + p}du = \frac{1}{2} B\Big(z,\frac{n-1}{2} + p\Big) = \frac{\Gamma(\frac{n-1}{2} + p)}{2}\ \frac{\Gamma(z)}{\Gamma(z + \frac{n-1}{2} + p)}.
    \end{align}
The $\sigma^m$ coefficient of \eqref{eq:B-expansion} are determined by the poles of the product $\Gamma(z)\ \Gamma(z + \frac{n-1}{2}+p)^{-1}$. For any $n$, the numerator $\Gamma(z)$ has no zeros and only has poles and at $z\in-\N_0$ while $\Gamma(z + \frac{n-1}{2} + p)^{-1}$ has no poles and only simple zeros at $z\in-p-\frac{n-1}{2}-\N_0$. If $n$ is odd, $\frac{n-1}{2}$ is integer and every simple zero at $-p-\frac{n-1}{2}-\N_0$ meets a simple pole from $-\N_0$, leaving only simple poles at $z\in\{0,-1,...,-p-\frac{n-3}{2}\}$. When $n$ is even, every $z\in-\N_0$ is a simple pole while every $z\in-p-\frac{n-1}{2}-\N_0$ is a simple zero of the coefficient. 

\hspace{20pt} We now expand the measure $2ds$ as the power series near $\sigma = 0$. In either region $\pm s > 0$ near $s^2=1$ or $\sigma=0$, $\sigma = 1-s^2$ becomes $s = \frac{s}{|s|}\sqrt{1-\sigma}$ whose Taylor series brings 
    \begin{align}\label{eq:measure-expansion}
        2ds = \textstyle\frac{s}{|s|}(1-\sigma)^{-\frac{1}{2}}d\sigma=\displaystyle\frac{s}{|s|}\sum_{j = 0}^{\infty}\sigma^j\ 4^{-j}{2j-1\choose j-1}d\sigma.
    \end{align}
Substitute both \eqref{eq:measure-expansion} and \eqref{eq:B-expansion}, using $C_m(\frac{s}{|s|}\theta;z)$ as the coefficient of $\sigma^m$ in \eqref{eq:B-expansion}, and convolve the coefficients to find  
	\begin{align}\label{eq:B-coefficientss-defn}
		B\frac{2d\theta ds}{\sigma}&\sim\sum_{m = 0}^{\infty}\sigma^m\frac{s}{|s|}\underbrace{\sum^m_{j=0}4^{-(m-j)}{2(m-j)-1\choose m-j-1} C_j(\textstyle\frac{s}{|s|}\theta;z)}_{B_m(\frac{s}{|s|}\theta;z)}\displaystyle\frac{d\theta d\sigma}{\sigma}.
	\end{align}
Fix $m\in\N_0$. The coefficient $B_m(\theta;z)$ has at most the poles of $C_0$,...,$C_m$ and depends on the parity of $n$ and given by \eqref{eq:poles-B}. The coefficient $B_m$ is smooth in $\theta$ away from these poles as each $C_j$ is smooth in $\theta$ away from these poles. The coefficient $B_m$ has at least the zeros whenever $C_0,...,C_m$ has simultaneous zeros which if $n$ is odd doesn't occur but occur as simple zeros at $-m-\frac{n-1}{2}-\N_0$ when $n$ is even. Formulas \eqref{eq:B_0-most-sing} emerge from \eqref{eq:B-coefficientss-defn} and \eqref{eq:beta-func-sum-1.1} and \eqref{eq:B-expansion} after setting all indices $m = j = p = 0$. In particular, the zeroth coefficient of \eqref{eq:fiber-integral-expansion} is $H_0(\frac{s}{|s|}\theta) = h(\frac{s}{|s|}\theta)\ \omega_{n-2}$.
\end{proof}

\begin{lem}\label{PairingPoleStructure}
Following the hypothesis of Proposition \ref{MellinPairing}, for each $m\in\N_0$ there exists meromorphic $A_{1,m}(z)$ and $A_{2,m}(z)$ and a function $\tilde R_m(z)$ which is analytic in $\Real(z) > -\Real(\gamma) -\frac{n-1}{2} - m$ away from $-\N_0$ and having poles at $-\N_0$ such that
	\begin{align}\label{eq:pairing-expansion}
		\langle h,\mu_M\rangle_{\partial\Omega} = \sum_{q=0}^{m-1} A_{1,q}(z)\ A_{2,q}(z) + \tilde R_m(z).
	\end{align}
The function $A_{1,q}$ has a unique pole of order $(l+1)$ at $z = -(\frac{n-1}{2} + \gamma + q)$ and $A_{2,q}$ has at most simple poles at $z\in\Lambda_{n,q}$ defined by \eqref{eq:poles-B}. When $n$ is even, $A_{2,q}(z)$ has simple zeros at $z\in-\frac{1}{2}- \frac{n-2}{2} - q-\N_0$.
\end{lem}
\begin{proof}
From formula \eqref{eq:pairing-prefinal}, the support $\supp(\chi(1\mp s))$ ensures we are in one of the two regions defined by $\pm s > 0$. In that region, expansion \eqref{eq:B-measure-expansion} of $B\frac{2d\theta ds}{\sigma}$ from Lemma \ref{BDensityExpansion} holds. The expansion is equivalent to the existence of $R_m(s,\theta;z)$, meromorphic in $z$ with poles at most at $-\N_0$, smooth on $\overline{Z}$ away from $z\in-\N_0$, and on locally compact subsets $K\subset\C\setminus-\N_0$, $|R_m|\leq C_K\sigma^{m}$, such that for each $m\in\N_0$
	\begin{align}\label{eq:B-expansion-finite}
		B(s,\theta;z)\frac{2d\theta ds}{\sigma} = \frac{s}{|s|}\sum_{q=0}^{m-1}\sigma^q\ B_q(\textstyle\frac{s}{|s|}\theta;z)\displaystyle\frac{d\theta d\sigma}{\sigma} + R_m(s,\theta;z)\frac{s}{|s|}\frac{d\theta d\sigma}{\sigma}.
	\end{align}
Submit the expansion \eqref{eq:B-expansion-finite} (the one appropriate for the region) into formula \eqref{eq:pairing-prefinal} and change the order of the sum and integration and separate to find 
	\begin{align}\label{eq:defn-A1-A2}
		\frac{1}{2}\sum_{q=0}^{m-1}\frac{s}{|s|}\underbrace{\int_{\sigma=0}^1\sigma^{\frac{n-1}{2} + z + \gamma + q}\log(\sigma)^{\el}\tilde\chi(\sigma)\frac{d\sigma}{\sigma}}_{A_{1,q}(z)}\ \underbrace{\int_{S^{n-1}}a(\theta)\ \textstyle B_q(\frac{s}{|s|}\theta;z)d\theta}_{A_{2,q}(z)} \ +\ \tilde R_m(z),
	\end{align}
where $\tilde\chi(\sigma) = \chi(1 - \sqrt{1-\sigma})$ remains a smooth cutoff in $\sigma$ under the coordinate change noting that $\tilde\chi\equiv0$ in a neighborhood of $\sigma=1$, and $\tilde R_m(z)$ is formula \eqref{eq:pairing-prefinal} with $B$ replaced by $R_m$.

\hspace{20pt} For each $q$, $A_{1,q}$ is a Mellin functional of a simple component $\sigma^w\log(\sigma)^{\el}$ and defines a meromorphic function with a unique pole at $z = -w$ of order $(\el+1)$, where $w = \frac{n-1}{2} + z + \gamma + q$ (see Appendix \ref{MellinBackground}). For each $q$, $A_{2,q}$ adopts the poles and zeros estimates of $B_q(\frac{s}{|s|}\theta;z)$ (described in Lemma \ref{BDensityExpansion}) as $a(\theta)B_q(\frac{s}{|s|}\theta;z)$ is smooth in $\theta$ and $S^{n-1}$ is compact. The remainder $\tilde R_m(z)$ is  
	\begin{align}\label{eq:remainder}
		\tilde R_m(z) = \frac{s}{|s|}\int_{\sigma=0}^1\sigma^{\frac{n-1}{2} + z + \gamma}\log(\sigma)^{\el}\tilde\chi(\sigma)\int_{S^{n-1}}a(\theta)\ R_m(s,\theta;z)d\theta\frac{d\sigma}{\sigma}.
	\end{align}
The inner integral defines a smooth function of $s$ in the region $\pm s > 0$ and adopts the meromorphic structure of $R_m$ since $a(\theta)R_m(s,\theta;z)$ are smooth and $S^{n-1}$ compact. To show the analyticity claim about $\tilde R_m$, let $K\subset\{z\in\C : \Real(z) > -\Real(\gamma) - \frac{n-1}{2} - t\}\setminus-\N_0$ be compact. On $K$, we have a conormal bound $|R_m|\leq C\sigma^m$ and 
	\begin{align}
		|\tilde R_m|\leq C\int_{\sigma=0}^1\sigma^{\frac{n-1}{2} + \Real(z + \gamma) + m}\log(\sigma)^{\el}\tilde\chi(\sigma)\frac{ds}{\sigma},
	\end{align}
where $C$ may be negative depending on the parity of $\el$, is an upper bound that is uniformly bounded on $K$ since $\frac{n-1}{2} + \Real(z+ \gamma) + m > 0$ and absolute convergence shows \eqref{eq:remainder} defines an analytic function on $K$.
\end{proof}

We finally present the proof of Theorem \ref{MainResult}. The proof follows the analysis of the kinds of interactions the poles and zeros of $A_{1,m}(z)$ and $A_{2,m}(z)$ from Lemma \ref{PairingPoleStructure} have and the consequences of those corresponding interactions on the index set side.

\begin{proof}[Proof of Theorem \ref{MainResult}]
The indices of $R^*u$ correspond to the poles and their orders of $\langle h,(R^*u)_M\rangle_{\partial\Omega}$. The phg function $R^*u$ supports an index $(\gamma,k)$ only if the meromorphic family $(R^*u\frac{dx}{\rho})_M$ supports an order $k+1$ pole at $z=-\gamma$, see subsection \ref{MellinBackground}. From Proposition \ref{PairingPoleStructure}, the pole structure of $\langle h,(R^*u)_M\rangle_{\partial\Omega}$ is known up to examination of products of meromorphic functions. In all cases, the estimate on the pole support is found by summing the estimated pole support of $A_{1,m}\cdot A_{2,m}$ for each $m$.

\hspace{20pt} We begin with case \textbf{d.} Assume for all $m\in\N_0$ that the estimated pole support of $A_{1,m}$ does not meet estimated zero support or pole support of $A_{2,m}$. This is the generic case when the extended union in Theorem \ref{thm:pushforward} is no larger than the usual union. We begin by taking $n\geq2$ to be \textbf{even}. $A_{1,m}$ has at most simple poles at $-\N_0$ while $A_{2,m}$ has a unique order $\el+1$ pole at $z = -(\frac{n-1}{2}+\gamma + m)$. Summing these poles over $m\in\N_0$ finds a pole estimate corresponding to the index set
	\begin{align}\label{eq:case-d-est}
		F_{(\gamma,\el)} = \textstyle\overline{\big\{(\frac{n-1}{2} + \gamma,\el)\big\}}\;\cup\;\overline{\big\{(0,0)\big\}}.
	\end{align}
For $n\geq2$ which is \textbf{odd}, for each $m\in\N_0$, $A_{2,m}$ has at most simple poles supported at $\{0,-1,...,-\frac{n-3}{2} - m\}$. Summing these, this pole support corresponds to the same upper estimate \eqref{eq:case-d-est}.

\hspace{20pt} \textbf{Case a.} Suppose for some $m$, $A_{1,m}$ supports a pole located on the support of a zero of $A_{2,m}$. This requires that $n$ be even as $A_{2,m}$ has no consistent vanishing when $n$ is odd. A pole at $z = -(\frac{n-1}{2} + \gamma + m)$ meets a zero of $A_{2,m}$ only if $(\frac{n-1}{2} + \gamma + m)\in-(\frac{1}{2} + \frac{n-2}{2} + m + \N_0)$ which occurs if and only if $\gamma\in\N_0$. In particular, overlap for some $m$ implies overlap for all $m$. Let $\gamma = q\in\N_0$. The product $A_{1,m}\cdot A_{2,m}$ will have an order $\el$ pole at $z = -(\frac{n-1}{2} + q + m)$ in addition to the simple poles supported on $-\N_0$ by $A_{2,m}$. Summing these estimates over $m\in\N_0$, the collection corresponds to the index set 
	\begin{align}
		F_{(\gamma,l)} = \textstyle\overline{\big\{(\frac{1}{2} + \frac{n-2}{2} + q, \el-1)\big\}}\cup\overline{\big\{(0,0)\big\}},
	\end{align}
where the initial part of the union is empty when $l = 0$.

\hspace{20pt} \textbf{Case b.} Suppose $n\geq2$ is \textbf{even} and for some $m$, $A_{2,m}$ and $A_{1,m}$ support a pole at the same location in $\C$. $A_{1,m}$ supports a pole at $z = -(\frac{n-1}{2} + \gamma + m)$ and meet the simple poles supported at $-\N_0$ if and only if $\gamma\in\frac{1}{2} - \frac{n}{2} - m +\N_0$. A weaker but not sufficient condition is $-(\frac{n-1}{2} + \gamma + m)\in\Z$ and is equivalent to $\gamma \in\frac{1}{2} + \Z$ since $\frac{n}{2} + m$ is integer. Write $\gamma = \frac{1}{2} + q$ for $q\in\Z$. Returning to the original condition, $\frac{n}{2} + q + m\in\N_0$ implies overlap at index $(\frac{n}{2}+q+m,\el)$. When $\frac{n}{2} + q < 0$, every nonnegative integer $k$ has a nonnegative solution $m$ to $k = \frac{n}{2} + q + m$ and overlap occurs every index in $\overline{\big\{(0,\el)\big\}}$ but not at the indices in $\big\{(\frac{n}{2} + q,\el),...,(-1,\el)\big\}$ and the index set is 
    \begin{align}
        F_{(\gamma,\el)} = \textstyle\overline{\big\{(\frac{n}{2} + q,\el),(0,\el+1)\big\}}, && \text{when} && \frac{n}{2} + q<0.
    \end{align}
When $\frac{n}{2} + q \geq0$, overlap occurs at indices in $\overline{\big\{(\frac{n}{2} + q,\el)\big\}}$ but not before, as nonnegative integers strictly smaller than $\frac{n}{2} + q$ cannot be expressed as the sum of $\frac{n}{2} + q$ and a nonnegative integer. Accounting for the indices associated to the simple poles of each $A_{2,m}$ at each $z\in-\N_0$, the corresponding index set in this case is
    \begin{align}
        F_{(\gamma,\el)} = \textstyle\overline{\big\{(\frac{n}{2} + q,\el+1)\big\}}\cup\overline{\big\{(0,0)\big\}}, && \text{when} && \frac{n}{2} + q\geq0.
    \end{align}

\hspace{20pt} \textbf{Case c.} Suppose $n\geq2$ is \textbf{odd} and for some $m$, both $A_{1,m}$ and $A_{2,m}$ support a pole at $z = -(\frac{n-1}{2} + \gamma + m)$. $A_{2,m}$ only supports poles at $\{0,-1,...,-\frac{n-3}{2}-m\}$. As $\frac{n-1}{2} + m$ is always integer, we can restrict our attention to $\gamma\in\Z$. Pole overlap requires 
	\begin{align}
		0\leq\frac{n-1}{2} + m + \gamma\leq\frac{n-3}{2} + m\iff -\frac{n-1}{2} - m\leq \gamma\leq-1
	\end{align}
indicating overlap occurs when $\gamma =q\in-\N$. For any such $q$, overlap occurs for sufficiently large $m$ when $-\frac{n-1}{2} - q\leq m$ and for $-\frac{n-1}{2} - q \leq 0$, overlap occurs for all $m$ and accounting for the poles of $A_{2,m}$ supported at $z\in-\N_0$ the corresponding index set is 
    \begin{align}
        F_{(\gamma,\el)} = \textstyle\overline{\big\{(\frac{n-1}{2} + q,\el+1)\big\}}\cup\overline{\big\{(0,0)\big\}}, && \text{when} && \frac{n-1}{2} + q\geq0.
    \end{align}
When $-\frac{n-1}{2} - q > 0$, overlap starts at the smallest $m$ which meets $-\frac{n-1}{2} - q\leq m$. For such $m$, $\frac{n-1}{2} + q + m = 0$ indicating that overlap occurs at all indices in $\overline{\big\{(0,\el)\big\}}$ as $m$ is succeeded. In this case, overlap will not occur at indices $\big\{(\frac{n-1}{2} + q,
\el),...,(-1,\el)\big\}$ and the index set estimate is 
    \begin{align}
        F_{(\gamma,\el)} = \textstyle\overline{\big\{(\frac{n-1}{2} + q,\el),(0,\el+1)\big\}}, && \text{when} && \frac{n-1}{2} + q < 0.
    \end{align}
since this is the most singular term of the Laurent expansion. 
\end{proof}

\begin{proof}[Proof of Corollary \ref{cor:lowerbound}]
We study the coefficient of the most singular term of the index set $\overline{\{(\frac{n-1}{2} +\gamma,p)\}}\subseteq F_{(\gamma,\el)}$, with $p$ depending on $(\gamma,\el)$. The phg expansion of $R^*u$ supports components at these indices if and only if the coefficient of the most singular term is non-vanishing. From the analysis of the Mellin functional of a simple component $u=\tilde a(x)\log(\rho)^p\rho^{\frac{n-1}{2}+\gamma}\ dx$ in Appendix \ref{MellinBackground}, the coefficient $\tilde a$ is related to the coefficient of the most singular term in the Laurent expansion of $u_M$ by 
    \begin{align}\label{eq:coefficient-extract}
        \tilde a(x)\ d\omega = \frac{2}{(-1)^{p}\ p!}\lim_{z\rightarrow-\frac{n-1}{2}-\gamma}\Big(z+\frac{n-1}{2}+\gamma\Big)^{p+1}\ \langle h,\mu_M\rangle(z),
    \end{align}
where $d\omega$ is the round metric on $\partial\Omega$. In each case, only the $q=0$ component of \eqref{eq:pairing-expansion} supports this pole. Choose $m=1$ in \eqref{eq:pairing-expansion} to write 
    \begin{align}\label{eq:coeff-expansion}
        \langle h,\mu_M\rangle_{\partial\Omega} = A_{1,0}(z)\ A_{2,0}(z) + \tilde R_1(z),
    \end{align}
where $\tilde R_1$ is analytic in a neighborhood of $z=-(\gamma+\frac{n-1}{2})$. From \eqref{eq:MellinGeneral} and \eqref{eq:defn-A1-A2} we can write 
    \begin{align}\label{eq:coeff-A1-A2}
        A_{1,0}(z) = \frac{(-1)^{\el} \el!}{(z + \frac{n-1}{2} + \gamma)^{\el+1}} + V(z), && A_{2,0}(z) = \int_{S^{n-1}}a(\textstyle\frac{s}{|s|}\theta) B_0(\theta;z)d\theta,
    \end{align}
after changing variable in $A_{2,0}$ to $\theta' = \frac{s}{|s|}\theta$, which is well-defined near $\sigma=0$. Combine \eqref{eq:coeff-A1-A2} and \eqref{eq:coeff-expansion} to write
    \begin{align}\label{eq:pair-isolated}
        \langle h,\mu_M\rangle_{\partial\Omega} = (-1)^{\el} \el!\frac{A_{2,0}(z)}{(z + \frac{n-1}{2} + \gamma)^{\el+1}} + \tilde V(z)
    \end{align}
for some $\tilde V$ which is bounded in a neighborhood of $z=-(\frac{n-1}{2} + \gamma)$. The coefficient $B_0$ is found by \eqref{eq:B_0-most-sing} to express
    \begin{align}\label{eq:coeff-A2-final}
        A_{2,0}(z) = \langle h,S_{\frac{s}{|s|}}a\rangle\ \displaystyle\frac{\omega_{n-2}}{2}\ \Gamma\Big(\frac{n-1}{2}\Big)\ \begin{cases}
            \displaystyle\frac{\Gamma(z)}{\Gamma(z + \frac{n-1}{2})}, & n\text{ even,}\\
            \displaystyle\prod_{k=0}^{\frac{n-3}{2}}\frac{1}{z+k}, & n\text{ odd.}
        \end{cases}
    \end{align}
Where $S_{\frac{s}{|s|}}$ is the operator on $C^{\infty}(S^{n-1})$ sending $a\in C^{\infty}(S^{n-1})$ to $S_{\frac{s}{|s|}}a(\theta) = a(\frac{s}{|s|}\theta)$. We follow the same order of casework.

\hspace{35pt}\textbf{Case d.} We assume $\frac{n-1}{2}+\gamma$ is not integer, or $n$ is even and $\gamma\notin\N_0$, or that $n=2$ is odd and $\gamma\notin-\N$. In this case $p=\el$ in \eqref{eq:coefficient-extract} and combined with \eqref{eq:pair-isolated} finds
    \begin{align}\label{eq:cased-coefficient-3}
       c_d= \frac{2}{(-1)^{\el} \el!}\lim_{z\rightarrow-(\gamma + \frac{n-1}{2})}\Big(z + \frac{n-1}{2} + \gamma\Big)^{\el+1}\ \langle h,\mu_M\rangle_{\partial\Omega} = 2\ A_{2,0}\bigg(-\gamma-\frac{n-1}{2}\bigg).
    \end{align}
Evaluating \eqref{eq:coeff-A2-final} at $z = -(\frac{n-1}{2} + \gamma)$ finds 
    \begin{align}\label{eq:coeff-case-d}
        c_d = \langle h,S_{\frac{s}{|s|}}a\rangle\ \displaystyle\omega_{n-2}\ \Gamma\Big(\frac{n-1}{2}\Big)\ \begin{cases}
            \displaystyle\frac{\Gamma(-\gamma-\frac{n-1}{2})}{\Gamma(-\gamma)}, & n\text{ even,}\\
            \displaystyle\prod_{k=0}^{\frac{n-3}{2}}\frac{1}{-\frac{n-1}{2}-\gamma+k}, & n\text{ odd.}
        \end{cases}
    \end{align}
When $n$ is even, $\gamma\notin\N_0$ implies $0<|\Gamma(-\gamma)|<\infty$ and $\gamma + \frac{n-1}{2}$ being non-integer implies $0<|\Gamma(-\frac{n-1}{2}-\gamma)|<\infty$ showing $c_d\in\C\setminus\{0\}$ in this case. When $n$ is odd, $\gamma\notin-\N$ implies $|\gamma + \tilde k| > 0$ for each $\tilde k = \frac{n-1}{2}-k = 1,2,...,\frac{n-1}{2}$ showing that $c_d\in\C\setminus\{0\}$.

\hspace{35pt}\textbf{Case a.} We assume even $n$ and $\gamma\in\N_0$. In this pole annihilation case, $p=\el-1$ in \eqref{eq:coefficient-extract} and \eqref{eq:pair-isolated} finds
    \begin{align}\label{eq:coeff-case-a}
        c_a = \frac{2}{(-1)^{\el-1} (\el-1)!}\lim_{z\rightarrow-(\frac{n-1}{2} + \gamma)}\Big(z + \frac{n-1}{2} + \gamma\Big)^{\el}\ \langle h,\mu_M\rangle_{\partial\Omega} = -2\el\lim_{z\rightarrow-(\frac{n-1}{2} + \gamma)}\frac{A_{2,0}(z)}{(z + \frac{n-1}{2} + \gamma)}.
    \end{align}
From the even $n$ case of \eqref{eq:coeff-A2-final} combined with \eqref{eq:coeff-case-a} is
   \begin{align}\label{eq:coeff-case-a-2}
        c_a = -\el\ \langle h,S_{\frac{s}{|s|}}a\rangle\ \omega_{n-2}\ \Gamma\Big(\frac{n-1}{2}\Big)\ \Gamma\Big(-\frac{n-1}{2}-\gamma\Big)\ \displaystyle\lim_{z+\frac{n-1}{2}\rightarrow-\gamma}\frac{\Gamma(z+\frac{n-1}{2})^{-1}}{(z + \frac{n-1}{2} + \gamma)},
    \end{align}
where $0<|\Gamma(-\frac{n-1}{2}-\gamma)|<\infty$ since $\frac{n-1}{2}+\gamma$ is non-integer. The function $\Gamma(\xi)$ has a simple pole at $\xi = z+\frac{n-1}{2} = -\gamma\in-\N_0$ with residue $\Res(\Gamma;-\gamma) = \frac{(-1)^{\gamma}}{\gamma!}$. Inserting this residue into \eqref{eq:coeff-case-a-2} 
   \begin{align}\label{eq:coeff-case-a-2}
        c_a = \el\ \langle h,S_{\frac{s}{|s|}}a\rangle\ \omega_{n-2}\ \Gamma\Big(\frac{n-1}{2}\Big)\ \Gamma\Big(-\frac{n-1}{2}-\gamma\Big)\ (-1)^{\gamma+1}\gamma!
    \end{align}
This shows $c_a\in\C\setminus\{0\}$, when $\el > 0$.

\hspace{20pt} \textbf{Case b.} We take $n$ even and $\gamma\in\frac{1}{2}+\Z$. We have $p=\el+1$ if $\frac{n-1}{2} + \gamma\geq0$ in \eqref{eq:coefficient-extract} and $p=\el$ if $\frac{n-1}{2} + \gamma < 0$. With \eqref{eq:pair-isolated}, we find
    \begin{align}\label{eq:case-b-lim}
        c_b &= \lim_{z\rightarrow-(\gamma + \frac{n-1}{2})}\Big(z + \frac{n-1}{2} + \gamma\Big)^{p+1}\ \frac{2\langle h,\mu_M\rangle_{\partial\Omega}}{(-1)^{p} p!}\\
        &= \langle h,S_{\frac{s}{|s|}}a\rangle\ \omega_{n-2}\ \frac{\Gamma(\frac{n-1}{2})}{\Gamma(-\gamma)}\lim_{z\rightarrow-(\frac{n-1}{2} +\gamma)}\begin{cases}
            \Gamma(z), & \frac{n-1}{2} +\gamma < 0,\\
            \frac{(z + \frac{n-1}{2} +\gamma)\ \Gamma(z)}{-(\el+1)}, & \frac{n-1}{2} +\gamma\geq0,
        \end{cases}
    \end{align}
where $0<|\Gamma(-\gamma)|<\infty$ since $\gamma\in\frac{1}{2}+\Z$. When $\frac{n-1}{2} + \gamma$ is a nonnegative integer, $\Gamma(z)$ has residue $\Res(\Gamma;-\frac{n-1}{2} - \gamma) = \frac{(-1)^{\frac{n-1}{2} + \gamma}}{(\frac{n-1}{2} + \gamma)!}$ finding \eqref{eq:case-b-lim} to evaluate to 
    \begin{align}
        c_b &= \langle h,S_{\frac{s}{|s|}}a\rangle\ \omega_{n-2}\ \frac{\Gamma(\frac{n-1}{2})}{\Gamma(-\gamma)}\ \begin{cases}
            \Gamma(-\frac{n-1}{2} -\gamma), & \frac{n-1}{2} +\gamma < 0,\\
            \frac{(-1)^{\frac{n-1}{2} + \gamma}}{-(\frac{n-1}{2} + \gamma)!\ (\el+1)}, & \frac{n-1}{2} +\gamma\geq0.
        \end{cases}
    \end{align}

\textbf{Case c.} We take odd $n\geq2$ and $\gamma\in-\N$. Like in case b, $p=\el$ when $\frac{n-1}{2} + \gamma < 0$ or $p=\el+1$ when $\frac{n-1}{2} + \gamma\geq0$ respectively in \eqref{eq:coefficient-extract}. With \eqref{eq:pair-isolated}, find
    \begin{align}\label{eq:case-c-lim}
        c_c &= \lim_{z\rightarrow-(\gamma + \frac{n-1}{2})}\Big(z + \frac{n-1}{2} + \gamma\Big)^{p+1}\ \frac{2\langle h,\mu_M\rangle_{\partial\Omega}}{(-1)^{p} p!}\\
        &= \langle h,S_{\frac{s}{|s|}}a\rangle\ \omega_{n-2}\ \Gamma\Big(\frac{n-1}{2}\Big)\lim_{z\rightarrow-(\frac{n-1}{2} +\gamma)}\begin{cases}
            \prod_{k=0}^{\frac{n-3}{2}}\frac{1}{z+k}, & \frac{n-1}{2} +\gamma < 0,\\
            \frac{(z + \frac{n-1}{2} +\gamma)}{-(\el+1)}\prod_{k=0}^{\frac{n-3}{2}}\frac{1}{z+k}, & \frac{n-1}{2} +\gamma\geq0.
        \end{cases}
    \end{align}

When $\frac{n-1}{2} + \gamma\geq0$, $\frac{n-3}{2}\geq\frac{n-1}{2} + \gamma\geq0$ since $\gamma\in-\N$ and $(z +\frac{n-1}{2} + \gamma)$ does appear as a factor in the product. We can evaluate \eqref{eq:case-c-lim} to 
    \begin{align}
        c_c &= \langle h,S_{\frac{s}{|s|}}a\rangle\ \omega_{n-2}\ \Gamma\Big(\frac{n-1}{2}\Big)\begin{cases}
            \prod_{k=0}^{\frac{n-3}{2}}\frac{1}{-\frac{n-1}{2} -\gamma + k}, & \frac{n-1}{2} +\gamma < 0,\\
            \frac{-1}{\el+1}{\prod'}_{k=0}^{\frac{n-3}{2}}\frac{1}{-\frac{n-1}{2} -\gamma +k}, & \frac{n-1}{2} +\gamma\geq0,
        \end{cases}
    \end{align}

where $\prod'$ notates the product with only $\frac{n-3}{2}$-many factors and forgetting the factor corresponding to $z=-(\frac{n-1}{2}+\gamma)$.
\end{proof}

\newpage

\appendix

\section{Appendix: Background Material}

\subsection{Manifolds with Corners}\label{ManifoldWithCorners}

A manifold with corners $X$ (mwc) of dimension $d$, like $G$ (see Section \ref{SectionDevelopment}), is a Hausdorff, second-countable, topological space $X$ where each point $p\in X$ has a neighborhood homeomorphic to an open subset of $[0,\infty)^d$ for which the homeomorphisms define a smooth atlas. The \textit{codimension} of $p\in M$ is defined as the number of coordinates which vanish at $p$ in a coordinate neighborhood $U\subset[0,\infty)^d$ (identifying $p$ with its image under the homeomorphism). The coodinates are sometimes called \textit{adapted coordinates}. A manifold with boundary (mwb) is a manifold with corners for which $\coDim\leq1$. We will only be interested in the case when $\coDim\leq 2$. The interior is $X^{int} = \coDim^{-1}\{0\}$. Non-interior points form the manifold boundary $\partial X = \coDim^{-1}(\N)$ which is a $(d-1)$-dimensional manifold with corners. The boundary $\partial X$ is thought of as a disjoint union of open, positive codimensional `boundary subfaces' of $X$. 

\hspace{20pt} A \textit{boundary hypersurface} (bhs) of $X$ is the closure (in the topology of $X$) of a maximal, connected, open $(d-1)$-dimensional submanifold of $\partial X$ for which the closure is embedded in $X$ as a $(d-1)$-dimensional mwc. Each bhs $W$ is associated to a class of smooth \textit{boundary-defining function} (bdf) $\rho\in C^{\infty}\big(X,[0,\infty)\big)$ which satisfy 
    \begin{enumerate}
        \item $\rho^{-1}(0) = W$,
        \item $d\rho\neq0$ on $W$ and
        \item $d\rho_{\alpha}(x),d\rho_{\beta}(x)\in T^*_xX$ are linearly independent for each $x\in W_{\alpha}\cap W_{\beta}$.
    \end{enumerate}

Condition $2.$ ensures $W$ is embedded, while condition $3.$ ensures the bhs intersect transversely (condition $3.$ is not the correct condition in higher codimension). A function $\sigma\in C^{\infty}\big(X,[0,\infty)\big)$ is a \textit{local bdf} of $W$ if there is an open neighborhood $U\supset W$ such that $\sigma|_U\in C^{\infty}\big(U,[0,\infty)\big)$ is a bdf for $W$. On an open $U\subset[0,\infty)^d$, each coordinate vanishing on $U$ defines is a bdf of the corresponding bhs.

\hspace{20pt} The $2$-codimensional boundary \textit{corners} are the $(d-2)$-dimensional connected components of the intersection of any $2$ distinct bhs, and similarly for higher codimension (these intersections are always transverse by the generalization of condition 3. above). The collection of $k$-codimensional bhs is denoted $M_k(X)$ and $M(X):=\cup_{k=1}^{\infty} M_k(X)\cup\{X\}$ of all faces forms a poset under inclusion. When $X$ is at most $k$-codimemsional, $M_{k}(X)$ consists of closed submanifolds of $X$ and $M_{k+j}(X)=\emptyset$ for all $j\in\N$. The inclusion poset for $G$ is 
	
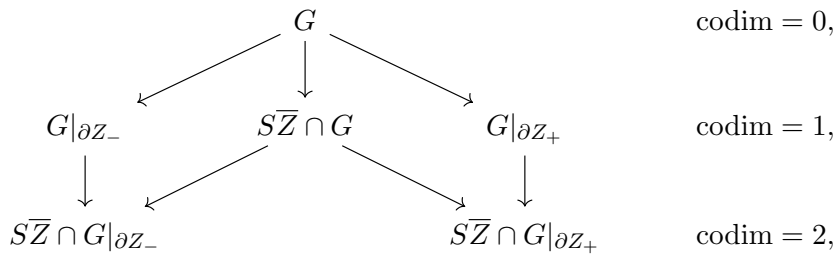
\begin{figure}[h]\centering
		\begin{tikzcd}
				& G\arrow[d]\arrow[dr]\arrow[dl] & & \coDim=0,\\
				G|_{\partial Z_-}\arrow[d] &  S\overline{Z}\cap G\arrow[dl]\arrow[dr] & G|_{\partial Z_+}\arrow[d] &  \coDim=1,\\
			  	S\overline{Z}\cap G|_{\partial Z_-} & & S\overline{Z}\cap G|_{\partial Z_+} &  \coDim=2,
		\end{tikzcd}
	\end{figure}
    
where $\partial Z_{\pm} = \{\pm s\}\times S^{n-1}$ and $S\overline{Z}$ denotes the tangent sphere bundle over $\overline{Z}$ with the induced Riemannian structure from Euclidean $\R\times\R^n\supset Z$. For examples of every term mentioned, see Section \ref{SectionDevelopment}. For reference, see \cite{Melrose1992,Grieser2001,Mazzeo2021}.

\subsection{Maps between Manifolds with Corners}\label{bMaps}

A function $f:X\rightarrow Y$ between mwc $X$ and $Y$ is a \textit{$b$-map} if $f|_{X^{int}}$ is smooth and there exists an $M_1(X)\times M_1(Y)$ matrix of nonnegative integers $e_f(W,V)$ such that for each collection $(r_V)_{V\in M_1(Y)}$ of bdf on $Y$, there are smooth, positive $\{g_{V}\}_{V\in M_1(Y)}$ on $X$ and a factorization
	\begin{align}
		f^*(r_V) = g_V\ \prod_{W\in M_1(X)}\rho_W^{e_f(W,V)}\qquad\text{or}\qquad f^*(r_V)\equiv0,\label{b-factorization}
	\end{align}
where $(\rho_W)_{W\in M_1(X)}$ is the family of bdf on $X$. The possibility $f^*(r_V)\equiv0$ corresponds to $X\subseteq f^{-1}(V)$ while the possibility $f^*(r_V) = g_V > 0$ corresponds to $X\subseteq f^{-1}(V^c)$. In any other case, at least one $e_f(W,V) > 0$ and both $\cup_{e_f(W,V) > 0}W\subseteq f^{-1}(V) $ and $X^{int}\cup_{e_f(W,V) = 0}W \subseteq f^{-1}(V^c)$. 

\hspace{20pt} A $b$-map $f$ such that for some $V\in M_1(Y)$, $f^*(r_V)\equiv0$ then $f$ is said to be \textit{non-interior} and is otherwise said to be \textit{interior}. When a $b$-map $f$ is interior, $e_f$ are called the \textit{geometric exponents} of $f$ and are invariant under change in adapted coordinates (compare equation \eqref{eq:smooth-invariance} in Appendix \ref{AlgebraOnManifold}). We now assume mwb is compact for the following definitions. A $b$-map $f:X\rightarrow Y$ is \textit{$b$-normal} if for all boundary faces $K\in M(X)$
	\begin{align}
		\coDim\big(f(K)^{int}\big)\leq\coDim(K^{int}).\label{b-normality}
	\end{align}
A $b$-normal map $f$ and is a \textit{$b$-fibration} if in addition, for all boundary faces $K\in M(X)$
	\begin{align}
		f|_{K^{int}}:K^{int}\rightarrow f(K)^{int}
	\end{align} 
is a fibration. In particular, the restriction is onto. For reference see \cite{Melrose1992,Grieser2001,Mazzeo2021}.

\subsection{Conormal and Polyhomogeneous Function Spaces}\label{AlgebraOnManifold}

For $G$, define the \textit{face function} $\Fa:G\rightarrow M(G)$ which assigns $p\in G$ to the unique face $F\in M(G)$ with $p\in F^{int}$ (interior taken in the subspace topology of $G$). A \textit{$b$-tangent vector field} on $G$ is a smooth vector field $V$ that is tangent to each embedded face i.e. $V_p\in T_p\Fa(p)$ for each $p\in G$. In coordinates on $U\subset G$, a smooth vector field $V$ is a $b$-vector field if for each $p\in G$, the pushforward of $V|_U$ lies in 
	\begin{align}\label{eq:b-vf}
		\Span_{C^{\infty}(U)}\bigg\langle x^1\, \frac{\partial}{\partial x^1},\frac{\partial}{\partial y^1},...,\frac{\partial}{\partial y^{n-1}}\bigg\rangle && \text{or} && \Span_{C^{\infty}(U)}\bigg\langle x^1\,\frac{\partial}{\partial x^1},x^2\,\frac{\partial}{\partial x^2},\frac{\partial}{\partial y^1},\ldots,\frac{\partial}{\partial y^{n-2}}\bigg\rangle,
	\end{align} 
if $\coDim(p) = 1$ or if $\coDim(p)=2$, respectively, when $G$ is at most $2$-codimensional. The $b$-vector fields form a Lie subalgebra $\mathcal{V}_b(G)\subset\Gamma(G)$ under the usual bracket and can be realized as the smooth sections of the \textit{$b$-vector bundle} $\prescript{b}{} TG$ which is not constant rank, see \cite{Melrose1992}.

\hspace{20pt} By a \textit{conormal} function we mean $u\in C^{\infty}(G^{int})$, and at least one associated conormal multi-weight $\mathfrak{t}:M_1(X)\rightarrow\R$, such that any derivation of $u$ under $b$-vector fields is locally uniformly $\MO(\rho^{\mathfrak{t}})$ where $\rho^{\mathfrak{t}} = \prod_{\alpha}\rho_{\alpha}^{t_{\alpha}}$. The \textit{conormal class} of functions with multi-weight $\mathfrak{t}$ is denoted
    \begin{align}
        \A^{\mathfrak{t}}(G) = \big\{u\in C^{\infty}(G^{int}) :\forall\{V_1,...,V_k\}\subset\mathcal{V}_b,  V_1\cdots V_ku = \mathcal{O}(\rho^{\mathfrak t})\big\},
    \end{align}
where $\mathcal{O}$ holds locally uniformly on $G$. The set $\A(G)$ denotes the union of all $\A^{\mathfrak{t}}(G)$. We work with a conormal subclass which is more amenable to analysis. First, define an \textit{index set} to be a discrete $E\subset\C\times\N_0$ satisfying 
	\begin{enumerate}[label = \arabic*., start=1]
		\item\label{indexset1} For each $t\in\R$, 
        the set $E_{< t} := \{(\gamma,l)\in E : \Real(\gamma)< t\}$ is finite
		\item\label{indexset2} For each $(\gamma,\el)\in E$ and $k\in\{0,1,...,\el\}$, $(\gamma,k)\in E$
		\item\label{indexset3} For each $(\gamma,\el)\in E$ and $n\in\N_0$, $(\gamma + n,\el)\in E$. 
	\end{enumerate}

A function $u\in\A(G^{int})$ is \textit{polyhomogeneous conormal} (phg) with respect to $W$ if there is at least one index set $E$, such that there are associated coefficient functions $(u_{\gamma,k})_{(\gamma,k)\in E}\subset C^{\infty}(W^{int})$ which for each $t\in\R$
	\begin{align}\label{asymptoticequivalence}
		u - \sum_{(\gamma,k)\in E_{<t}}u_{\gamma,l}(y)\ \rho^{\gamma} \log(\rho)^k\in\A^t\big([0,\epsilon)_{\rho}\times W^{int}_y\big),
	\end{align}
where we've identified a tubular neighborhood of $W^{int}$ in $G$ with the set $[0,\epsilon)\times W^{int}$ and $\rho$ is the normal coordinate and is extendable to a local bdf of $W$. Terms in the sum in \eqref{asymptoticequivalence} are called \textit{simple phg components}. For any cutoff $\chi\in C^{\infty}_c[0,\infty)$ with $\chi\equiv1$ near $0$, then $u\chi(\rho)$ will retain phg expansion \eqref{asymptoticequivalence} near $W^{int}$\footnote{$u$ and $u\chi$ share the same germ of smooth functions at $W^{int}$.}. Condition \ref{indexset1} ensures the sum in \eqref{asymptoticequivalence} is finite for every $t$. We write $u\sim\sum_{(\gamma,k)\in E}u_{\gamma,l}\ \rho^{\gamma} \log(\rho)^k$ as a shorthand for \eqref{asymptoticequivalence}.

\hspace{20pt} We say a function $u\in\A(G)$ is \textit{polyhomogeneous} on $G$ if there is at least one \textit{index family} $\E=(E_W)_{W\in M_1(G)}$ such that $u$ is polyhomogeneous with respect to $W$ with index set $E_W$ and the coefficients meet satisfy the following matching condition whenever $H\in M_1(G)\setminus\{W\}$ forms a $2$-codimensional face with $W$. Let $E,F\in\E$ denote the index sets associated to $W$ and $H$ respectively. For each $(\gamma,k)\in E$, each coefficient $u_{\gamma,k}\in C^{\infty}(W^{int})$ must be phg with respect to each connected component of $H\cap W$ with index set $F$ and similarly for each $(\eta,\el)\in F$, the coefficient $u_{\eta,\el}$ is phg with respect to each connected component of $H\cap W$ with index set $E$ and for each $(u_{\gamma,k})_{\eta,\el} = (u_{\eta,\el})_{\gamma,k}$, independent of the order of expansion. For example, if $u:[0,\infty)_{(x,y)}^2\rightarrow\C$ is a smooth function all the way to the boundary, then there are smooth $u_{x^k}(0,-):[0,\infty)_y\rightarrow\C$ and $u_{y^j}(-,0):[0,\infty)_x\rightarrow\C$ such that near either bhs, the $x$- or $y$-axis,
    \begin{align}
        u\sim\sum_{k=0}^{\infty}\frac{u_{x^k}(0,y)}{k!}x^k,\quad x\rightarrow0^+, && u\sim\sum_{j=0}^{\infty}\frac{u_{y^j}(x,0)}{j!}y^j,\quad y\rightarrow0^+
    \end{align}
in the sense of \eqref{asymptoticequivalence} by Taylor's Theorem and by expanding the coefficients as $xy\rightarrow0^+$, we have near the only $2$-codimenssional point $(0,0)$ the expansion
    \begin{align}
        u\sim\sum_{j=0}^{\infty}\sum_{k=0}^j\frac{{u_{y^{j-k}}}_{x^k}(0,0)}{(j-k)!k!}y^{j-k}x^k\sim\sum_{k=0}^{\infty}\sum_{j=0}^k\frac{{u_{x^{k-j}}}_{y^j}(0,0)}{(k-j)!j!}x^{k-j}y^j,\quad xy\rightarrow0^+.
    \end{align}
The matching condition is equivalent to the equality of mixed partial derivatives ${u_{y^j}}_{x^k}(0,0) = {u_{x^k}}_{y^j}(0,0)$. The space of phg $u$ with index family $\E = (E_{\alpha})_{W_{\alpha}\in M_1(G)}$ is denoted $\A_{\text{phg}}^{\E}(G)$ and we denote the union of these function spaces over every index family by $\A_{\text{phg}}(G)$.

\hspace{20pt} Conditions \eqref{indexset2} and \eqref{indexset3} for index sets arise from considering phg expansions under change of coordinates. If both $\rho$ and $\rho'$ are bdfs of the same $W\in M_1(G)$, then $\rho'/\rho:G\setminus W\rightarrow(0,\infty)$ extends to a positive, smooth function on $G$. For $(\gamma,\el)\in\C\times\N_0$
	\begin{align}\label{eq:smooth-invariance}
		\rho^{\gamma}\log(\rho)^{\el} = \Big(\rho'\frac{\rho}{\rho'}\Big)^{\gamma}\log\Big(\rho'\frac{\rho}{\rho'}\Big)^{\el} = \sum_{k=0}^{\el} A_{\gamma,k}(\rho')\ (\rho')^{\gamma}\log(\rho')^k,
	\end{align}
where $A_{\gamma,k}(\rho') = (\frac{\rho}{\rho'})^{\gamma}{\el\choose k}\log(\rho/\rho')^k\in C^{\infty}(G)$. By expanding $A_{\gamma,k}(\rho')$ as a Taylor expansion about $\rho'=0$, \eqref{eq:smooth-invariance}, both expressions are seen to have the same index set $\overline{\big\{(\gamma,\el)\big\}}$.

\hspace{20pt} Define the \textit{most singular index} of an index set $E$ as the real number $\inf(E) = \min\{\Real(\gamma) : (\gamma,\el)\in E\}$. The number $\inf(E)$ provides conormal estimates i.e. $u\in\A^{E}_{\text{phg}}$ implies $u\in\A^{\inf(E)-\epsilon}$ for each $\epsilon > 0$. For reference see \cite{Melrose1992,Grieser2001,Mazzeo2021}.

\subsection{The Pullback and Pushforward Theorems}\label{PullbackPushforwardTheorems}

Each $b$-maps $f:X\rightarrow Y$ induces a linear maps $f^*:\A_{\text{phg}}(Y)\rightarrow\A_{\text{phg}}(X)$-spaces by pullback. The Pullback theorem \cite{Melrose1992} estimates the index set of $f^*u$ from the index set of $u$ and the geometric exponents of $f$. Let $\tilde e_f$ denote the matrix 

\begin{thmlit}[Pullback Theorem]\label{thm:pullback}
Let $X$ and $Y$ be manifolds with corners and $f:X\rightarrow Y$ a $b$-map with geometric exponents $e_f$. For each index family $\F$ on $Y$ define the index family $f^*\F$ on $X$ for $H\in M_1(X)$ by
	\begin{align}
		f^*\F(H) = \overline{\big\{(0,0)\big\}} + \sum_{\substack{K\in M_1(Y),\\ \ e_f(H,K) > 0}}\begin{pmatrix}e_f(H,K) & 0\\ 0 & 1\end{pmatrix}\cdot \F(K),
	\end{align}
where $\cdot$ is matrix multiplication, then $f^*:\A^{\F}_{\text{phg}}(Y)\rightarrow\A^{f^*\F}_{\text{phg}}(X)$ is a linear map.
\end{thmlit}

\hspace{20pt} Each $b$-fibration induces a linear map $f_*:\prescript{b}{}\Omega(X)\rightarrow\prescript{b}{}\Omega(Y)$, up to an integrability condition, where $\prescript{b}{}\Omega(G)$ denotes the set of \textit{$b$-densities}, where a $b$-density is any density on $G^{int}$. In coordinates on $U\subset G$, we standard expression for $\mu\in\prescript{b}{}\Omega(G)$, near a $k$-codimensional face $F\in M(G)$ defined by the vanishing of bdf $\rho_1,...,\rho_k$, in the dual frame to \ref{eq:b-vf} is
    \begin{align}\label{eq:local-b-density}
        \mu = a\ \frac{d\rho_1\cdots d\rho_k\ \mu_F}{\rho_1\cdots\rho_k},
    \end{align}
where $\mu_F$ is a density on $F^{int}$ and $a$ is a smooth function on $U\cap G^{int}$. If there is an index family $\E$ such that near any $F\in M(G)$, the coefficient $a$ in \eqref{eq:local-b-density} has index family $\E|_F = \{E_W\in\E : F\subset W\}$, then we say $\mu$ is a phg $b$-density and said to be $\A_{\text{phg}}^{\E}$. This is well-defined by the coordinate invariance of index sets, see Appendix \ref{AlgebraOnManifold}.

\hspace{20pt} The \textit{pushforward} of a $b$-fibration is defined as the dual map to $f^*:C^{\infty}_c(Y^{int})\rightarrow C^{\infty}_c(X^{int})$ under the pairing of compactly-supported functions and densities $C^{\infty}_c(G)\times\prescript{b}{}\Omega(G)\rightarrow\C$ given by integration of the density $h\mu$ over $G$, for $(h,\mu)\in C^{\infty}_c(G)\times\prescript{b}{}\Omega(G)$. If $E$ and $F$ are two index sets, define the \textit{extended union} $E\overline{\cup} F$ as the superset of $E\cup F$ by
	\begin{align}
		E\,\overline{\cup}\,F = E\cup F\cup\big\{(\gamma, p + q + 1)\in\C\times\N_0 : (\gamma,p)\in E\text{ and }(\gamma,q)\in F\big\}.
	\end{align}
We restate the Pushforward theorem from \cite{Grieser2001}. 

\begin{thmlit}[Pushforward Theorem]\label{thm:pushforward}
Let $X$ and $Y$ be manifolds with corners and $f:X\rightarrow Y$ a $b$-fibration with geometric exponents $e_f$. For any index family $\E$ on $X$ such that $\inf(\E(H)) > 0$ whenever $e_f(H,-)\equiv0$, define $f_*\E$ on $Y$ for $K\in M_1(Y)$ by
	\begin{align}
		f_*\E(K) = \bigcup_{\substack{F\in M(X)\\ f(F)\subset K}}\tilde{\E}(F), && \tilde{\E}(F) = \overline{\bigcup}_{\substack{H\in\E|_F\\ e_f(H,L) > 0}}\begin{pmatrix}e_f(H,L) & 0\\ 0 & 1\end{pmatrix}^{-1}\cdot\E(H),
	\end{align}
where $\cdot$ is matrix multiplication, then $f_*\mu$ is $\A_{\text{phg}}^{f_*\E}$ whenever $\mu$ is $\A_{\text{phg}}^{\E}$.
\end{thmlit}

\begin{rmk}
A bhs $H\subset X$ with $e_f(H,\underline{\;\;\;})\equiv0$ fails to map to any boundary face of $Y$ under $f$. As a consequence, the singular behavior of a $b$-density $\mu$ near $H$ cannot be concentrated at the boundary $\partial Y$ and will not influence the index family of $f_*\mu$. The integrability condition $\inf(\E(H)) > 0$ ensures a $b$-density $\mu$ is integrable near $H$ so that $f_*\mu$ is well-defined on $Y^{int}$.    
\end{rmk}

For motivation for the appearance of the extended union, refer to \cite{GrieserGruber2001}.

\begin{prop}\label{Prop:ForwardEstimate}
Let $E$ be an index set with $\inf(E) > -1$ and $u\in\A^E(\overline{\Omega})$. Then $Ru\in\A^{\F}(\overline{Z})$ where
	\begin{align}
		\F(\partial Z_{\pm}) = E+\Big(\frac{n-1}{2},0\Big),
	\end{align}
and $\partial Z_{\pm} = \{\pm1\}\times S^{n-1}$.
\end{prop}
\begin{proof}
We start with equation \eqref{eq:RadonPullPush} from Proposition \ref{Prop:Pullback-Pushforward}. For $u\in\A^E(\overline{\Omega})$, $\hat\pi^*u\in\A^{\E'}(G)$ where on each bhs $W\in M_1(G)$, $\E'(W) = E$ by Theorem \ref{thm:pullback} as $e_{\hat\pi}\equiv1$ and $E+\N_0 = E$ by \ref{indexset3} The $b$-density $\mu = \hat\pi^*u\ dv_{(s,\theta)}d\theta ds$ has standard representation $\mu = r \sigma \ \hat\pi^* u\ \frac{dvd\theta ds}{r\ \sigma}$ with index set $E' = E+(1,0)$ for each bhs of $G$. The condition $\inf(E)>-1$ becomes $\inf(E')>0$ for each bhs, including $S\overline{Z}\cap G$ which has $e_{\hat P}(S\overline{Z}\cap G,-)\equiv 0$. Theorem \ref{thm:pushforward} applies and the $b$-density $\hat P_*\mu$ is $\A^{\F}(\overline{Z})$ where $\F(\partial Z_{\pm}) = E'$ because $e_{\hat P}(G|_{\partial Z_{\pm}},\partial Z_{\pm})=1$ and $e_{\hat P}(G|_{\partial Z_{\pm}},\partial Z_{\mp})=0$. 

\hspace{20pt} Writing equation \eqref{eq:RadonPullPush} as $R(u)\frac{d\theta ds}{\sigma} = \sigma^{\frac{n-3}{2}}\ \hat P_*(\mu)$, the index set $E'+(\frac{n-3}{2},0) = E+(\frac{n-1}{2},0)$ follows.
\end{proof}

\begin{prop}\label{Prop:Back-ProjectionEstimate}
Let $\E = (E_+,E_-)$ be an index family on $\partial Z$ and $w\in\A^{\E}(\overline{Z})$, then $R^*(w)\in\A^{F}(\overline{\Omega})$ where
	\begin{align}\label{eq:R*-std-est}
		F = \bigcup_{\pm}\bigg(\Big(\frac{n-1}{2},0\Big) + E_{\pm}\bigg)\;\overline{\cup}\;\overline{\big\{(0,0)\big\}}.
	\end{align}
\end{prop}
\begin{proof}
Start with equation \eqref{eq:BackProjPullPush}. Theorem \ref{thm:pullback} states $\hat P^*w\in\A^{\E}_{\text{phg}}(G)$, where $\E(G|_{\partial Z_{\pm}}) = E_{\pm}$, from $e_{\hat P}(G|_{\partial Z_{\pm}},\partial Z_{\pm})=1$ and $e_{\hat P}(G|_{\partial Z_{\pm}},\partial Z_{\mp})=0$, while $\E(S\overline{Z}\cap G) = \overline{\{(0,0)\}}$, as $e_{\hat P}(S\overline{Z}\cap G,-)\equiv 0$. Writing the $b$-density $\eta = \sigma^{\frac{n-1}{2}}\ \hat P^*w\ dv_{(s,\theta)}d\theta ds$ in standard form $\eta = \sigma^{\frac{n+1}{2}}\ r\ \hat P^*w\ \frac{dvd\theta ds}{r\ \sigma}$, $\eta$ has index family $\E(G|_{\partial Z_{\pm}}) = E_{\pm}+(\frac{n+1}{2},0)$ and $\E(S\overline{Z}\cap G) = \overline{\{(1,0)\}}$.

\hspace{20pt} The $b$-fibration $\hat\pi$ has $e_{\hat\pi}\equiv1$ and $\hat\pi_*\eta$ is well-defined and phg of class $\A_{\text{phg}}^{F'}$ by Theorem \ref{thm:pushforward}, where 
    \begin{align}
        F' = \bigcup_{\pm}\bigg(\Big(\frac{n+1}{2},0\Big) + E_{\pm}\bigg)\;\overline{\cup}\;\overline{\big\{(1,0)\big\}}.
    \end{align}
This is also the index set of $\rho R^*w$ when we write $\hat\pi_*\eta$ in standard form $\hat\pi_*\eta = \rho R^*w \frac{dx}{\rho}$, then identity $\big(E + (\gamma,0)\big)\overline{\cup}\big(F+(\gamma,0)\big) = (\gamma,0) + E\overline{\cup} F$ for $\gamma=1$ shows $R^*w$ has index set \eqref{eq:R*-std-est}.
\end{proof}

\subsection{Mellin Representation of a phg $b$-density}\label{MellinBackground}

Theorem \ref{MainResult} arises from studying the application of a Mellin functional to a phg $b$-density $R^*u\frac{dx}{\rho}$. Let $\mu$ be a phg $b$-density on a mwc $X$ with index set $E$ associated to $W\in M_1(X)$, $\rho$ a bdf associated to $W$, and $\chi\in C^{\infty}_c[0,\infty)$ a cutoff with $\chi\equiv1$ near $0$. The \textit{Mellin representation} $\mu_M(z)$ is a meromorphic family of $b$-densities on $W^{int}$ defined as the meromorphic extension of 
	\begin{align}\label{eq:Mellin-mu-2}
		\mu_M(z) = \int_{\rho = 0}^{\infty}\chi(\rho)\rho^z\ \mu,
	\end{align}
which converges locally uniformly on $\{z\in\C : \Real(z) > -\inf(E)\}$. The index set $E$ determines an upper estimate on the support and orders of the poles of $\mu_M$. We state this well-known result (see \cite{Mazzeo2021}).

\begin{thm}\label{Thm-Mellin}
Let $X,\rho,W$ be as above and $\chi\in C^{\infty}_c[0,\infty)$ a cutoff such that $\rho$ is a bdf on $\supp(\chi)$, $\supp(\chi)^{int}$ is contained in tubular neighborhood of $W^{int}$, $\chi\equiv1$ near $\rho=0$ and $\chi\equiv0$ near $\rho=\sup(\rho(U))$, and $\mu$ a $b$-density on $X$. If $E$ is an index set for $\mu$, then \eqref{eq:Mellin-mu-2} defines an analytic family of $b$-densities of $W$ on $\{z\in\C : \Real(z) > -\inf(E)\}$ and extends to a meromorphic family of $b$-densities on $W$ with at most order $\el+1$ poles at $z = -\gamma$ for each $(\gamma,\el)\in E$.
\end{thm}

In the particular case of Theorem \ref{Thm-Mellin} for the case of a simple phg component $\mu = u_{\gamma,\el}\ \rho^{\el}\log(\rho)^{\el}\frac{d\rho}{\rho}$ on the half-space $[0,\infty)$, $\mu_M$ has Laurent expansion at $z=-\gamma$
	\begin{align}\label{eq:MellinGeneral}
		\mu_M(z) = \frac{(-1)^{\el+1}\el!}{(z+\gamma)^{\el+1}} + \sum_{k\geq0}\frac{-\int_0^1\chi'(\rho)\ \log(\rho)^{k + \el + 1}d\rho}{k!\cdot(k+\el+1)}\ (z+\gamma)^k.
	\end{align}

\hspace{20pt} Theorem \ref{Thm-Mellin} enables the study of the phg support of $\mu$ by way of studying the poles of its Mellin representation $\mu_M$. The estimate in Theorem \ref{Thm-Mellin} goes the other way as well, if $\mu_M$ fails to have an order $k$ pole at $z=\eta-n$ for any $n\in\N_0$, then $\mu$ has an index family $\E$ for which $\E(W)$ doesn't contain $(-\eta,k-1)$. We restate Theorem \ref{Thm-Mellin} for the particular case of a simple phg component $\tilde a\ \rho^{\gamma}\log(\rho^{\el})\frac{dx}{\rho}$ on $\overline{\Omega}$ for our use. 

\begin{cor}\label{Cor-Mellin}
Let $(\gamma,p)\in\C\times\N_0$, $a\in C^{\infty}(\partial\Omega)$. If $\mu = \tilde a(x)\rho^z\log(\rho)^p\frac{dx}{\rho}$, then for each $K\in\N_0$, \eqref{eq:Mellin-mu-2} is 
    \begin{align}
        \frac{\mu_M}{d\omega} =  \frac{a(\omega)}{2}\sum_{k=0}^{K-1}a_k\int_0^1\chi\rho^z\ \big(\rho^{\gamma+k}\log(\rho)^{p}\frac{d\rho}{\rho}\big) + R_K(\omega;z),
    \end{align}
where $R_K$ is smooth in $\omega$ and analytic on $\{z\in\C : \Real(z) > -\Real(\gamma)-K\}$, $\{a_k\}_{k=0}^{\infty}$ are the Taylor coefficients of $(1-\rho)^{\frac{n-2}{2}}$ near $\rho=0$ and $d\omega$ is the round measure on $\partial\Omega$.
\end{cor}

\textbf{Acknowledgements} The author thanks Fran\c{c}ois Monard for many discussions and mentoring. The author was partially supported by NSF-CAREER grant DMS-1943580.

\end{document}